\documentclass[12pt,a4paper]{article}
\usepackage{amsfonts, amsthm, amsmath, amssymb}
\newtheorem{theorem}{Theorem}[section]
 \newtheorem{corollary}[theorem]{Corollary}
 \newtheorem{lemma}[theorem]{Lemma}
 \newtheorem{proposition}[theorem]{Proposition}
 \theoremstyle{definition}
 \newtheorem{definition}[theorem]{Definition}
 \theoremstyle{remark}
 \newtheorem{remark}[theorem]{Remark}
 \newtheorem*{example}{Example}
 \numberwithin{equation}{section}

\DeclareMathOperator{\Alg}{\mathsf{Alg}}
\DeclareMathOperator{\Cyc}{\mathsf{Cyc}}
\DeclareMathOperator{\diag}{diag}
\DeclareMathOperator{\disc}{disc}
\DeclareMathOperator{\Hyplat}{\mathsf{HypLat}}
\DeclareMathOperator{\Lat}{\mathsf{Lat}}
\DeclareMathOperator{\myspan}{span}
\def \one {\boldsymbol{1}}
\DeclareMathOperator{\ran}{ran}
\DeclareMathOperator{\rank}{rank}
\DeclareMathOperator{\srank}{\mbox{$\ast$}-rank}
\def \Re{}
\DeclareMathOperator{\slim}{s-lim}
\DeclareMathOperator{\sdet}{\mbox{$\ast$}-det}

\begin{document}

\title {On the Spectral Analysis of Direct Sums \\
  of Riemann-Liouville Operators\\
  in  Sobolev Spaces of Vector Functions}

%


\maketitle
\begin{abstract}
Let $J_k^\alpha$  be a real power of the integration operator
$J_k$ defined on Sobolev space $W_p^k[0,1]$. We investigate the
spectral properties of the operator $A_k=\bigoplus_{j=1}^n \lambda_j
J_k^\alpha$ defined on $\bigoplus_{j=1}^n W_p^k[0,1]$.\ Namely, we
describe the commutant  $\{A_k\}'$, the double commutant
$\{A_k\}''$ and the algebra  $\Alg A_k$. Moreover, we describe the
lattices $\Lat A_k$ and $\Hyplat A_k$ of invariant and
hyperinvariant subspaces of $A_k$, respectively. We also calculate
the spectral multiplicity $\mu_{A_k}$ of $A_k$ and describe the
set $\Cyc A_k$ of its cyclic subspaces.
In passing, we present a simple counterexample for the implication
\begin{equation*}
\Hyplat(A\oplus B)=\Hyplat A\oplus \Hyplat B\Rightarrow
\Lat(A\oplus B)=\Lat A\oplus \Lat B
\end{equation*}
to be valid.
\end{abstract}

\section
{Introduction}
 \label{intro}
It is well known
 \cite{M.S.Brodskii, I.C.Gohberg and M.G.Krein,B. Sz.-Nagy and C. Foias 1, N.K.Nikolskii}
that the Volterra integration operator $J:\ f(x)\to \int_0^xf(t)\,dt$ as well as its real powers $J^\alpha$
 play an exceptional role in the spectral theory of nonselfadjoint operators in $L_2[0,1]$. The paper is devoted to the spectral analysis of direct sums of multiples of powers $J^\alpha$ of the integration operator $J$ in Sobolev spaces.
To describe its content we first briefly recall basic facts on the integration operator.

It is well known
 \cite{M.S.Brodskii, I.C.Gohberg and M.G.Krein, B. Sz.-Nagy and C. Foias 1, N.K.Nikolskii}
that  $J$  is unicellular on $L_p[0,1]$  for $p\in [1,\infty)$
and the  lattice $\Lat J$ of its invariant subspaces is
anti-isomorphic to the segment $[0,1]$.
  The same is also true (see
  \cite{I.C.Gohberg and M.G.Krein,N.K.Nikolskii})
  for the
  simplest Volterra operators
  \begin{equation*}
     J^{\alpha}:\ f(x)\to
     \int_0^x\frac{(x-t)^{\alpha-1}}
     {\Gamma(\alpha)}f(t)\,dt,\qquad \Re \alpha >0,
  \end{equation*}
  being the positive powers  of the integration operator $J$.

More precisely, it is known  (see \cite{M.S.Brodskii,I.C.Gohberg
and M.G.Krein,B. Sz.-Nagy and C. Foias 1,N.K.Nikolskii}) that
\begin{equation}\label{1point1}
\begin{split}
     \Lat  J^\alpha &=
     \Hyplat   J^\alpha=
     \{E_a:\  a\in [0,1]\},\\
     E_a:&= \{ f\in L_p[0,1] : \ f(x)=0 \ \text{ for a.a. }
\ x\in[0,a]\}.
\end{split}
\end{equation}
Description \eqref{1point1} yields (and, in fact, is equivalent to)
\cite{M.S.Brodskii,I.C.Gohberg and M.G.Krein,N.K.Nikolskii}  the
following description of cyclic vectors of $J^\alpha$
\begin{equation}
f \text{   is a cyclic vector for   }  J^\alpha\Leftrightarrow
      \int_0^\varepsilon |f(x)|^p\,dx>0\ \ \ \
      \text{ for all }\ \varepsilon >0.
      \label{equ5.1}
\end{equation}
This condition is called the $\varepsilon$ - condition.

Description \eqref{1point1} of $\Hyplat   J^\alpha$ is closely
connected with the description of the commutant $\{J^\alpha\}'$. The
commutant $\{J\}'$ of the operator $J$  defined on $L_2[0,1]$
 as well as the (weakly closed) algebra $\Alg J$ generated by $J$
and $\mathbb I$  were originally  described by D. Sarason \cite{D.
Sarason1} (see also a simple proof in \cite{J.A.Erdos}). Another,
description of $\Alg J$ for $J$ acting in $L_p[0,1]$ has also been
obtained in \cite{M.M.Malamud 1,preprint MMM}. Namely, it
was shown in \cite{M.M.Malamud 1,preprint MMM} that if $J$
is defined on $L_p[0,1]$ ($1< p<\infty$), then
$\{J^\alpha\}^{\prime}=\Alg J^\alpha$ and
$K\in\{J^\alpha\}^{\prime}$ if and only if it is bounded
and admits a representation
    \begin{equation}
             (Kf)(x)=
             \frac{d}{dx}\int_0^xk(x-t)f(t)\,dt,\ \ \
             \ k\in L_{p'}[0,1],\label{convoper}
      \end{equation}
where ${p'}^{-1}+p^{-1}=1$. Using a criterion of boundedness of
$K$ defined  on $L_2[0,1]$  (see  \cite[Proposition 3.1']{preprint
MMM}) it can easily be shown that for $p=2$ description
\eqref{convoper} is equivalent to that obtained in \cite{D.
Sarason1}.

 Now, let $ A=J^\alpha\otimes B(=\bigoplus_{j=1}^n\lambda_jJ^\alpha)$ be
a tensor product of the operator $J^\alpha$ defined on $L_p[0,1]$
and the $n\times n$ nonsingular diagonal matrix
  $B=\diag (\lambda_1,\dots,\lambda_n )\in\mathbb C^{n\times n}$.
The investigation of such operators with $B= B^*$ was initiated by G.
Kalisch \cite{G.K.Kalisch}. He has extended the known  Livsic
theorem (see  \cite{M.S.Brodskii,I.C.Gohberg and
M.G.Krein})  to the case of (abstract) Volterra operators with
finite-dimensional real part and characterized those of them that
are unitarily equivalent to $A$ with $B= B^*$ and $\alpha=1$ (see
also \cite{M.S.Brodskii, I.C.Gohberg and M.G.Krein}).

Later on, sufficient conditions for a Volterra operator\\ $K: f\to
\int_0^xK(x,t)f(t)\,dt$ defined on  $L_p[0,1]\otimes \mathbb C^{n}$
to be similar to the
 operator $A$ have been indicated in \cite{M.M.Malamud 5}. So, $A$
 may be treated as a similarity model for a wide
 class of Volterra operators. This result has been applied in \cite{M.M.Malamud 5}
to the problem of unique recovery of a Dirac type system by its monodromy matrix (see also
references therein).

 Further, one of the authors  \cite{M.M.Malamud 1,M.M.Malamud 2}
described the lattices $\Lat  A$ and $\Hyplat   A$ and the set
$\Cyc  A$ of cyclic subspaces of the operator $ A=J^\alpha\otimes
B(=\bigoplus_{j=1}^n\lambda_jJ^\alpha)$ defined on $L_p[0,1]\otimes
\mathbb C^n,\ p\in (1,\infty )$. In particular, in
\cite{M.M.Malamud 1,M.M.Malamud 2} necessary and sufficient
conditions for a sequence $\{\lambda_i\}_{i=1}^n$ guaranteeing the
splitting of each of the  lattices $\Lat A$ and  $\Hyplat A$, as
well as of the commutant $\{A\}^{\prime}$ and double commutant
$\{A\}^{\prime\prime}$ of  $A$ were found. More precisely, it was
proved in \cite{M.M.Malamud 1,M.M.Malamud 2} that each of the
following relations
   \begin{align}
\Lat \bigoplus_{j=1}^n\lambda_jJ^\alpha&= \bigoplus_{j=1}^n\Lat
\lambda_jJ^\alpha,\label{equ1.3}\\
\Hyplat \bigoplus_{j=1}^n\lambda_jJ^\alpha&=
\bigoplus_{j=1}^n\Hyplat
\lambda_jJ^\alpha,\label{equ1.4}\\
\biggl\{\bigoplus_{j=1}^n\lambda_jJ^\alpha\biggr\}'=\bigoplus_{j=1}^n\{\lambda_jJ^\alpha\}'&=
\biggl\{\bigoplus_{j=1}^n\lambda_jJ^\alpha\biggr\}''=\bigoplus_{j=1}^n\{\lambda_jJ^\alpha\}''\label{equ1.6}
\end{align}
is equivalent to the condition
\begin{equation}
     \arg\lambda_i\ne \arg\lambda_j\pmod{2\pi}
     \ \qquad  1\leqslant i< j\leqslant n.
     \label{equ1.2}
\end{equation}

Some partial cases of the equivalence
\eqref{equ1.3}$\Leftrightarrow$\eqref{equ1.2} have been obtained
earlier in \cite{L.T.Hill, B.P. Osilenker and
             V.S.Shulman, B.P. Osilenker and V.S.Shulman2}
(see Remark \ref{OsilenkerShulman}).

It is easily seen that \eqref{equ1.6} is equivalent to the
following fact : for any  $\lambda\not\in (0,+\infty)$ an operator
equation
\begin{equation}
J^\alpha X=\lambda XJ^\alpha\label{1.8A}
     \end{equation}
has only zero bounded solution $X$. Moreover, in \cite{M.M.Malamud
1, M.M.Malamud 2} a description of all nonzero solutions $X$ of
\eqref{1.8A}  with $\lambda\in (0,+\infty)$ was obtained.
Recently, equation \eqref{1.8A}, and even more general ones with a
bounded $A$ in place of $J^\alpha$, has attracted attention of
several mathematicians (see, for instance, \cite{A.Biswas and
A.Lambert and S.Petrovic, A.Biswas and S.Petrovic, M.T.Karaev},
and \cite{Paul S. Bourdon and Joel H. Shapiro, John B. Conway and
Gabriel Prajitura, Shkarin2}).
 In particular, some results from
\cite{M.M.Malamud 1} on equation \eqref{1.8A} were rediscovered in
\cite{A.Biswas and A.Lambert and S.Petrovic} and \cite{M.T.Karaev}
(the case  $\alpha=1$)  and in
 \cite{A.Biswas and S.Petrovic} (the case  $\alpha\in \mathbb Z_+\setminus\{0\}$).
  These authors treat  any solution $X$ of $AX=\lambda XA$ as
an extended eigenvector of $A$ (see Remark \ref{remee} (2)).

Note also that if \eqref{equ1.2} is not fulfilled then $A$ is not
cyclic. The set $\Cyc A$ of cyclic subspaces of $A$ was described
in \cite{M.M.Malamud 1,M.M.Malamud 2} by using a notion of
$*$-determinant (see Definition \ref{Def5.2}). For example,
vectors $f_1:=(f_{11},f_{12})$, $f_2:=(f_{21},f_{22})$ generate a
cyclic subspace of the operator $A=J\oplus J$ defined on
$L_p[0,1]\oplus L_p[0,1]$ if and only if the function
$\sdet\left(\begin{matrix}f_{11}&f_{12}\\
f_{21}&f_{22}\end{matrix}\right):=f_{11}*f_{22}-f_{12}*f_{21}$
satisfies $\varepsilon$ - condition \eqref{equ5.1} (here $f*g$
stands for the convolution of functions  $f,\ g\in
L_1[0,1]\ :\ (f*g)(x):=\int_0^x f(x-t) g(t)\,dt$).

Passing to the case of the Sobolev space we should mention the
pioneering work of E. Tsekanovskii  \cite{E.R.Tsekanovskii}.  More
precisely, it is shown in \cite{E.R.Tsekanovskii} (see also
\cite{P.V.Ostapenko and V.G.Tarasov}) that the integration
operator $J_k\ :\ f(x)\rightarrow \int_0^x f(t)\,dt$ defined on
$W_p^k[0,1]$
  is unicellular too and $\Lat  J_k$ consists of continuous part $\Lat ^c J_k$
  and discrete part $\Lat ^dJ_k$, $\Lat  J_k=\Lat ^c J_k\cup \Lat ^d J_k$. Here
\begin{equation}\label{neweq9}
 \begin{split}
\Lat ^c J_k=&\bigl\{E_{a,0}^k:\ a\in (0,1]\bigr\}\cup E_{0,0},\quad\\
E_{a,0}^k:=\bigl\{f\in W_p^k[0,1]:\  &f(x)=0\ \text{ for } x
\in[0,a]\bigr\},\qquad E_{0,0}:= W_{p,0}^k[0,1],
\end{split}
\end{equation}
is a continuous chain and $\Lat ^d J_k=\{E_l^k\}_{l=0}^k$ with
$E_k^k:=W_p^k[0,1]$ and
   \begin{equation}
      E_l^k=\bigl\{f\in W_p^k[0,1] :\
      f(0)=\dots =f^{(k-l-1)}(0)=0\bigr\},\quad
       l\in\{1,\dots,k-1\},
      \label{equ1.7}
   \end{equation}
  is a discrete chain.
\noindent
 It is clear that, for $0\leqslant a_1\leqslant
a_2\leqslant 1$,
\begin{align*}
\{0\}&=E_{1,0}^k\subset E_{a_2,0}^k\subset E_{a_1,0}^k\subset
E_{0,0}^k\\
&=W_{p,0}^k[0,1]=E_0^k\subset E_1^k\subset\dots \subset
E_k^k=W_p^k[0,1].
\end{align*}
 In \cite{I.Yu.Domanov and M.M.Malamud} we investigated
 the spectral properties of the complex powers
$J_k^\alpha$ of the integration operator $J_k$ defined on Sobolev
space $W_p^k[0,1]$.
  Namely, in \cite{I.Yu.Domanov and M.M.Malamud} were described
  the lattices $\Lat  J_k^\alpha$ and $\Hyplat   J_k^\alpha$, the set of cyclic subspaces $\Cyc
  J_k^\alpha$, the  operator algebra $\Alg  J_k^\alpha$, the commutant $\{J_k^\alpha\}^{\prime}$
and the double commutant $\{J_k^\alpha\}^{\prime\prime}$. In
particular,
  it turns out that $\{J_k^\alpha\}'=\{J_k^\alpha\}''$ and $\{J_k^\alpha\}'$ and $\Alg J_k^\alpha$  can be  described as follows:
\begin{gather}
 R\in \{J_k^\alpha\}'\Leftrightarrow
(Rf)(x)=cf(x)+\int_0^xr(x-t)f(t)\,dt,\ \ \ r\in
  W_p^{k-1}[0,1],\label{commutantsobolev}\\
    \begin{split}
   R\in &\Alg J_k^\alpha\\
    \Leftrightarrow &\begin{cases}
    R\in \{J_k^\alpha\}',
   \ r^{(l)}(0)=0,\ \ l\ne m\alpha-1,
                \ m\leqslant [\frac{k-1}{\alpha}],& 1\leqslant \alpha\leqslant k-1,        \\
        R\in \{J_k^\alpha\}',\ r\in
  W_{p,0}^{k-1}[0,1],& 2\leqslant k\leqslant \alpha+\frac 1p.
  \end{cases}
  \end{split}
\label{algsobolev}
\end{gather}
  It was also shown in \cite{I.Yu.Domanov and M.M.Malamud} that the operator $J_k^\alpha$ is unicellular
  on $W_p^k[0,1]$ if and only if either $k=1$ or $\alpha=1$. Moreover, the unicellularity of $J_k^\alpha$
  is equivalent to the validity of the "Neumann-Sarason"
  identity  $\Alg  J_k^\alpha=\{J_k^\alpha\}^{\prime\prime}$.

In this paper  we extend the main results from \cite{I.Yu.Domanov
and M.M.Malamud} and \cite{M.M.Malamud 2} to the case of the
operator $A_k:=J_k^\alpha\otimes B$ defined on Sobolev space
$W_p^k[0,1]\otimes\mathbb C^n$ of vector-functions. Moreover, we
investigate the spectral properties of the operator
$A_k:=\bigoplus_{j=1}^{n}\lambda_jJ_{k_j}^\alpha$.

The paper is organized as follows.
 In Section
\ref{Preliminaries}, we  collect some auxiliary results about
invariant subspaces for $C_0$ contractions and accretive
operators. Here we also present  and  complete  some results from
\cite{M.M.Malamud 2} for the operator
$A=\bigoplus_{j=1}^n\lambda_jJ^\alpha$ defined on $\bigoplus_{j=1}^n
L_p[0,1]$.

In Section \ref{OperatorAk0}, it is shown that the  operator
$A=\bigoplus_{j=1}^n\lambda_jJ^\alpha$  defined on\\
 $\bigoplus_{j=1}^n L_p[0,1]$ and the operator
$A_{k,0}=\bigoplus_{j=1}^n\lambda_jJ_{k,0}^\alpha$ defined on
$\bigoplus_{j=1}^n W_{p,0}^k[0,1]$ are isometrically equivalent.
Hence all results on the operator $A$ presented in  Section
\ref{Preliminaries} are immediately extended to the case of
the operator $A_{k,0}$.

In Section \ref{OperatorAk}, we provide a spectral analysis of the
operator $A_k=\bigoplus_{j=1}^n\lambda_jJ_{k,0}^\alpha$ defined on
$\bigoplus_{j=1}^n W_p^k[0,1]$.
A descriptions  of the (weakly closed) algebra $\Alg  A_k$, commutant
$\{A_k\}^{\prime}$ and  double commutant $\{A_k\}^{\prime\prime}$
is presented in Subsection \ref{sec7}, Subsection \ref{sec6} and
Subsection \ref{sec777}, respectively.

In Subsection \ref{subsec4.2}, we obtain a description of the
lattice $\Lat A_k$ assuming that $A_k:=\bigoplus_{j=1}^n\lambda_j
J_{k_j}^\alpha$ satisfies condition \eqref{equ1.2}.
  This description is essentially based on a description of $\Lat
  T$ (Theorem \ref{theorem4.2})
  for  finite-dimensional operator $T$ in $\bigoplus_{j=1}^n\mathbb C^{k_j}$.
In Subsection \ref{hyperinvAk}, a description of the lattice
$\Hyplat A_k$ is contained.  We emphasize that $\Hyplat
A_{k,0}=\Hyplat ^c A_k$ and the "continuous part"
   of $\Hyplat   A_k$ does not depend on $\alpha$.

It turns out that
 under condition \eqref{equ1.2} $\Hyplat   A_k$ as well as the commutant
$\{A_k\}^{\prime}$ of the operator $A_k$ splits, that is, relations
\eqref{equ1.4}-\eqref{equ1.6} remain valid with $\Hyplat  A$ and
$\{A\}^{\prime}$ replaced by $\Hyplat A_k$ and $\{A_k\}^{\prime}$,
respectively.
 On the other hand, under condition \eqref{equ1.2} $\Lat  A_k$
does not split for $k\geqslant 1$ in contrast to \eqref{equ1.3}.

In this connection we recall (see \cite{J.B.Conway and B. P.Y.Wu})
that for  a direct sum $T_1\oplus T_2$ of two operators on a Banach
space the relations \eqref{equ1.4}-\eqref{equ1.6} are equivalent
to each other and both are implied by \eqref{equ1.3}.  Thus, the
operator $A_k$   presents a simple counterexample to the
validity of the implication
   \begin{equation*}
      \Hyplat  (T_1\oplus T_2)=\Hyplat  T_1\oplus \Hyplat   T_2\ \ \ \Longrightarrow
      \Lat (T_1\oplus T_2)=\Lat T_1\oplus \Lat  T_2 .
   \end{equation*}
Other counterexamples can be found in \cite{J.B.Conway and B.
P.Y.Wu}.

In Subsection \ref{sec5}, we compute the spectral multiplicity
and present a description of the cyclic subspaces $\Cyc A_{k}$ for
the operator $A_{k}$.

It should be emphasized that descriptions of the sets $\Cyc A_{k}$
and $\Cyc A_{k,0}$ essentially differ. Namely, the first
description does not depend on a choice of a sequence $\{\lambda_j\}_1^n$,
though the second one depends on $\{\arg\lambda_j\}_1^n$ and is
similar to that obtained in \cite{M.M.Malamud 2} for
$\bigoplus_{j=1}^n L_p[0,1]$.

A description of the  set of cyclic subspaces of the operator $
A=\bigoplus_{j=1}^m\lambda_j J_{k_j}^\alpha\oplus\bigoplus_{j=m+1}^n
\lambda_jJ_{k_j,0}^\alpha$ acting in the mixed space $
\bigoplus_{j=1}^m W_p^{k_j}[0,1]\oplus\bigoplus_{j=m+1}^n
W_{p,0}^{k_j}[0,1]$  is presented too.

Main results of the paper have been announced (without proofs) in
 \cite{I.Yu.Domanov and M.M.Malamud 2}.

\subsection
{
 Notations and agreements
}
\begin{enumerate}
\item  $X,X_1,X_2$ stand for Banach spaces;
\item $[X_1,X_2]$ is the space of
    bounded linear operators from $X_1$ to $X_2$; $[X]:=[X,X]$;
\item $\mathbb I$ and $\mathbb I_k$  denote the identity
operators on $X$ and on $\mathbb C^k$, respectively;
    $\mathbb O:=0\cdot\mathbb I$,
    $\mathbb O_k:=0\cdot\mathbb I_k$;
\item  $J(0;k)$ denotes the Jordan nilpotent cell of order $k$;
\item  $\ker   T=\{x\in X\ :\ Tx=0\}$  is the kernel of $T\in [X]$;
\item $\ran T=\{Tx\ :\ x\in X\}$ is the range of $T\in[X]$;
\item $\Cyc T$ denotes the set of cyclic subspaces of an operator
$T\in[X]$ (see Definition \ref{definitioncyclic});
\item $\{T\}^{\prime}$ and $\{T\}^{\prime\prime}$ denote the
commutant and the double commutant ( or bicommutant) of an
operator $T\in[X]$, respectively;
\item $\Alg \{T_1,\dots, T_n\}$ stands for a weakly closed subalgebra of $[X]$
generated by $T_1,\dots,T_n\in [X]$  and the identity $\mathbb I$;
\item $\Lat \mathcal{A}$ denotes the lattice of invariant subspaces
of the algebra $\mathcal{A}$;
\item  $\Lat T$ ($:=\Lat(\Alg T)$) and $\Hyplat T$
($:=\Lat(\{T\}')$) denote the lattices of invariant
and hyperinvariant subspaces of $T\in[X]$, respectively;
\item  $\myspan   E$ is the closed linear span of the set $E\subset
X$;
\item $r*f$ stands for the convolution of functions
    $r,\ f\in L_1[0,1]\ :\ (r*f)(x):=\int_0^x r(x-t) f(t)\,dt$;
\item $ \mathbb Z_+ := \{n\in \mathbb Z:\ n\geqslant 0 \}$; $ \mathbb R_+ := \{ x\in \mathbb R:\ x\geqslant 0 \}$.
\end{enumerate}
As usual, $W_p^k[0,1]$ $(p\in (1,\infty),\ k\in \mathbb
Z_+\setminus\{0\})$ stands for the Sobolev space consisting of
functions $f$ having $k-1$ absolutely continuous derivatives and
$f^{(k)}\in L_p[0,1]$.
  $W_p^k[0,1]$ is a Banach space equipped with the norm
  \begin{equation*}
      \|f\|_{W_p^k[0,1]}=
      \left[
      \sum_{j=0}^{k-1} |f^{ (j)} (0) |^p
      +\int_0^1 |f^{ (k)} (t) |^p\, dt
      \right] ^{1/p}.
  \end{equation*}
  $
  W_{p,0}^k[0,1]:=\{f\in W_p^k[0,1] :\ f(0)=\dots =f^{(k-1)}(0)=0\}
  $.

We set $W_p^0[0,1]:=L_p[0,1]$ and $W_{p,0}^0[0,1]=L_p[0,1]$.

Let $J_{k,0}^\alpha$ and $J_k^\alpha:= J_{k,k}^\alpha$ stand for  the
operator $J^\alpha$ defined on  $W_{p,0}^k[0,1]$ and $W_p^k[0,1]$,
respectively. The operator $J_{k,0}^\alpha$ is well defined on
$W_{p,0}^k[0,1]$ for any $\Re \alpha >0$. The operator
$J_k^\alpha$ is well defined on $W_p^k[0,1]$ if either $
\alpha\in \mathbb Z_+\setminus\{0\}$ or $\alpha
>k-\frac{1}{p}$. Therefore throughout the paper we assume that
\begin{enumerate}

\item the operator $A:=\bigoplus_{j=1}^n\lambda_j J^\alpha$ is
defined on $\bigoplus_{j=1}^n L_p[0,1]$ for $\Re \alpha >0$;
\item the operator  $A_{k,0}:=\bigoplus_{j=1}^n\lambda_j
J_{k_j,0}^\alpha$ is defined on $\bigoplus_{j=1}^n
W_{p,0}^{k_j}[0,1]$ with $k_j\geqslant 0$ and $\Re \alpha
>0$;
\item the operator  $A_k:=\bigoplus_{j=1}^n\lambda_j J_{k_j}^\alpha$
is defined on $\bigoplus_{j=1}^n W_{p}^{k_j}[0,1]$ with $k_j\geqslant
1$ and for $\alpha\in \mathbb Z_+\setminus\{0\}$ or $\alpha
>\max\limits_{1\leqslant j \leqslant n} k_j-\frac{1}{p}$.
\end{enumerate}
We  will also assume that $\lambda_j\ne 0$ for $j\in\{1,\dots,n\}$.
\section
{Preliminaries}
 \label{Preliminaries}
\subsection{Invariant
 subspaces  of some operators}
Here we present some known  results on invariant subspaces of  finite-dimensional
 nilpotent operators and $C_0$ contractions.We also
recall a condition about splitting of $\Alg
(A\oplus B)$, where $A,B\in [X]$.
\begin{theorem}\label{theorem4.2}
\cite{L.Brickman and P.A.Fillmore,I.Gohberg and P.Lancaster and
L.Rodman}
      If $Q$ is nilpotent on a finite-dimensional vector space V, then
      \begin{equation}
             \Lat (Q)
             =\bigcup_M\left\{[M,Q^{-1}M]:\
             M\in \Lat (Q\upharpoonright QV)\right\},
             \label{eqno(2.3)}
      \end{equation}
      where $[M,Q^{-1}M]$ is an interval in the lattice of all
      subspaces of $V$.
      Each interval satisfies the equation
      \begin{equation}
             \dim Q^{-1}M-\dim M=\dim\ker   Q.
      \label{eqno(2.4)}
      \end{equation}
\end{theorem}
The following result was first discovered by P. Halmos \cite{P.R.Halmos} for operators defined on  finite-dimensional spaces.
The generalization to $C_0$ contractions on Hilbert spaces
belongs to H. Bercovici \cite[Proposition 5.33]{Bercovici}, \cite[Corollary 2.11]{
Bercovici2} and P. Wu \cite[Theorem 1.2]{Wu}, and \cite[Theorem 5]{Wu2})(see also references therein).
     \begin{theorem}\label{Halmos}
Let $T$ be a $C_0$-contraction defined on  a separable Hilbert
space. Then every invariant subspace of $T$ is the closure of the
range and the kernel of some bounded linear transformation that
commutes with $T$,  that is,
          \begin{equation*}
                \Lat T=
                \bigl\{\ker   C : \  C\in\{T\}^{\prime}\bigr\}=
                \bigl\{\overline{\ran C}:\ C\in \{T\}^{\prime}\bigr\}.
          \end{equation*}
\end{theorem}
\begin{definition}(see \cite{B. Sz.-Nagy and C. Foias 1},\cite{N.K.Nikolskii})
Let $A$ and $B$ be  bounded operators defined on a Banach space $X_1$ and $X_2$ respectively.
$A$ is said to be quasisimilar to $B$ if there
exist deformations $K:\ X_1\rightarrow X_2$ and $L:\
X_2\rightarrow X_1$ (i.e. $\overline{\ran K}=X_2$, $\ker
K=\{0\}$, $\overline{\ran L}=X_1$, $\ker L=\{0\}$) such that
$AL=LB$ and $KA=BK$.
\end{definition}
\begin{remark}\label{remarkHalmos}
\begin{itemize}
\item[(i)]  Standard manipulations with Cayley transform implies
that Theorem \ref{Halmos} holds also for quasinilpotent accretive operators with
finite-dimensional real part.

\item[(ii)]  Let operator $A$ be  defined on a Banach space. Let also
$A$ be quasisimilar to a $C_0$ contraction $T$. Then, obviously
the statement of Theorem \ref{Halmos} is true for $A$, that is,
$\Lat A=\bigl\{\ker C : \ C\in\{A\}^{\prime}\bigr\}=
                \bigl\{\overline{\ran C}:\ C\in \{A\}^{\prime}\bigr\}$.
\end{itemize}
\end{remark}
Let $X$ be a Banach space and let $n$ be a positive integer. Then $X^{(n)}$ denotes the direct sum of $n$
copies of $X$. If $A$ is an operator on $X$, then $A^{(n)}$ denotes the direct sum of $n$ copies of $A$ (regarded as an operator on $X^{(n)}$).

The following theorem is
implicitly contained in  \cite{D. Sarason2} (see also \cite[Theorem 7.1,
Theorem 7.2]{H.Radjavi and P.Rosenthal} )
\begin{theorem}\label{corRadjRos}
Let $T_1,\dots, T_r\in [X]$ and
\begin{equation*}
\Lat( T_1^{(n)}\oplus\dots \oplus T_r^{(n)})=\Lat
T_1^{(n)}\oplus\dots \oplus\Lat T_r^{(n)},\qquad  n=1,2,\dots
\end{equation*}
Then $\Alg (T_1\oplus\dots \oplus T_r)=\Alg  T_1\oplus\dots
\oplus\Alg T_r$.
 \end{theorem}
\subsection{Spectral analysis of the
 operator $A=\bigoplus_{i=1}^n\lambda_iJ^\alpha$ defined on
$\bigoplus_{i=1}^n L_p[0,1]$}

 Throughout this subsection $X$ stands for $L_p[0,1]$, with $p\in
(1,\infty)$. Here we present some results from \cite{M.M.Malamud
2} on spectral analysis of the  operator
$A=\bigoplus_{i=1}^n\lambda_iJ^\alpha$ defined on $\bigoplus_{1}^n X$.
Moreover, we obtain a description of $\Alg A$ and investigate its
properties.

We begin with the following simple statement.
\begin{lemma}\label{lem2.6}
Let $A_i,M_i,N_i\in [X]$ for $i\in\{1,\dots,n\}$ and
$A=\bigoplus_{i=1}^n A_i$. Assume also that the following identities
are satisfied
\begin{equation}\label{ver35} A_i^m=M_iA_1^mN_i,\qquad
m\in \mathbb Z_+,  \qquad i\in\{1,\dots,n\}.
\end{equation}
Then
\begin{equation}\label{alglemma}
\Alg A= \biggl\{\bigoplus_{i=1}^n R_i :\ \ R_1\in \Alg A_1,\ \
R_i=M_iR_1N_i,\ \ i\in\{2,\dots,n\} \biggr\}.
\end{equation}
\end{lemma}
\begin{proof}
Let  $M:= \bigoplus_{i=1}^n M_i$ and  $N:=\bigoplus_{i=1}^n N_i$.
 Then  for any polynomial
$p(\cdot)$ identities \eqref{ver35}  yield  $p(A_i)=
M_ip(A_1)N_i.$  Hence,
\begin{equation*}
p(A)=\bigoplus_{i=1}^np(A_i)=\bigoplus_{i=1}^nM_ip(A_1)N_i=
M\biggl(\bigoplus_{i=1}^np(A_1)\biggr)N.
\end{equation*}
On the other hand, by definition of $\Alg A$ polynomials $p(A)$
are dense in $\Alg A$ in weak operator topology. Hence the last
identities imply $\Alg A=M\Alg(\bigoplus_{i=1}^nA_1)N$.  To complete
the proof it remains to note that $\Alg(\bigoplus_{i=1}^nA_1) =
\bigoplus_{i=1}^n\Alg(A_1).$
\end{proof}
Next  we apply Lemma  \ref{lem2.6} to describe $\Alg A$ for the
operator  $A=\bigoplus_{i=1}^n\lambda_iJ^\alpha$ with factors
$\lambda_i$ having equal arguments,
      \begin{equation}
      \lambda_i=\lambda_1/s_i^\alpha\
      ,\ \
      1=s_1\leqslant s_2\leqslant \ldots \leqslant s_n,\ \ \ \ \ \ \ \ i\in\{1,\dots,n\}.\label{arg=arg}
      \end{equation}
\begin{theorem}\label{th5.3alg}
      Let the operator
      $
      A=\bigoplus_{i=1}^n \lambda_iJ^\alpha
      $
      be defined on
      $
      \bigoplus_{i=1}^{n}X
      $
      with $\lambda_i$ satisfying  condition \eqref{arg=arg}.
      Then  $\Alg A$ is
      \begin{equation}
      \begin{split}
      \Alg A=\Bigl\{R=\diag(R_1,\dots,R_n) :
      (R_if)(x)=\frac{d}{dx}\int_0^xr_i(x-t)f(t)\,dt,\\
      r_1\in L_{p'}[0,1],\quad
 r_i(x)=r_1(s_i^{-1}x),\quad  R_1\in
      [L_p[0,1]]\Bigr\}.
\end{split}\label{Ri}
      \end{equation}
\end{theorem}
\begin{proof}
To apply Lemma  \ref{lem2.6} we introduce the operators $M_i$ and
$N_i$ by setting
\begin{equation}\label{neweq4new}
(M_if)(x): = f(s_i^{-1}x),\quad (N_if)(x):=\begin{cases}
f(s_ix),& x\in [0,s_i^{-1}],\\
0,&x\in [s_i^{-1},1].
\end{cases}\\
\end{equation}
Clearly, $\ker N_i=\{0\}$, $\ran N_i =\chi_{[0,s_i^{-1}]}L_p[0,1]
$, $\ker M_i= \chi_{[s_i^{-1},1]}L_p[0,1]$ and\\ $\ran M_i =
L_p[0,1].$ It can  easily be checked that $M_iN_i = I_{L_p[0,1]}$
and, moreover,
\begin{equation*}
(\lambda_iJ^\alpha)^m=M_i(\lambda_1J^\alpha)^m N_i,\qquad m\in
\mathbb Z_+, \qquad  i\in\{1,\dots,n\}.
\end{equation*}
Setting $A_i:= \lambda_iJ^\alpha$ and applying Lemma \ref{lem2.6}
we obtain
\begin{equation}\label{2.8}
       \Alg A=\Bigl\{R=\bigoplus_{i=1}^n R_i: R_1\in \Alg
      (\lambda_1J^\alpha),\quad R_i = M_iR_1N_i,\quad
      i\in\{2,\dots,n\}\Bigr\}.
\end{equation}
On the other hand, according to  \eqref{convoper}, any (bounded)
$R_1\in \Alg(\lambda_1J^\alpha)$ admits a representation
\begin{equation}
R_1\ :\ f(x)\rightarrow\frac{d}{dx}\int_0^xr_1(x-t)f(t)\,dt,\qquad
r_1\in L_{p'}[0,1].
\end{equation}
Straightforward  calculations show that that
\begin{equation*}
(M_iR_1N_if)(x)=\frac{d}{dx}\int_0^xr_1(s_i^{-1}(x-t))f(t)\,dt,\qquad
i\in\{2,\dots,n\}.
\end{equation*}
Combining the last equality with \eqref{2.8} we complete the
proof.
       \end{proof}
To state the results on  $\{A\}'$ we need some additional
notations. For any $a\in \mathbb R_+\backslash\{0\}$ we define an
operator $ L_a:\ X\rightarrow X $ by
\begin{equation} \label{neweq2}
      L_a:\ f(x)\to
                 g(x)=
                       \begin{cases}
                           f(ax),       &  0<a\leqslant 1, \\
                             \begin{cases}
                                   0,         & x\in [0,1-a^{-1}], \\
                                   f(ax-a+1), & x\in [1-a^{-1},1],
                             \end{cases} & a>1.
                       \end{cases}
\end{equation}
We set also
$$
L_a\{J^\alpha\}^{\prime}:=\{L_aK:\
K\in\{J^\alpha\}^{\prime}\},\qquad
\{J^\alpha\}^{\prime}L_a:=\{KL_a:\ K\in\{J^\alpha\}^{\prime}\}.
$$
It is easily checked that
$L_a\{J^\alpha\}^{\prime}=\{J^\alpha\}^{\prime}L_a$.
\begin{theorem}\cite[Proposition 4.6]{M.M.Malamud 2}
\label{pr6.3}
Suppose
      $
      A=\bigoplus_{i=1}^n \lambda_iJ^\alpha
      $
      is defined on
      $
      \bigoplus_{i=1}^{n}X
      $ and
      $\lambda_i$ satisfy condition \eqref{arg=arg}. Set also
      $a_{ij}:=s_i^{-1}s_j$ for $i,j\in\{1,\dots,n\}$.
      Then the commutant $\{A\}^{\prime}$ is of the form
\begin{equation*}
      \{A\}^{\prime}=
      \{
      K:\ K=(K_{ij})_{i,j=1}^n,\
      \ K_{ij}\in L_{a_{ij}}\{J^{\alpha}\}^{\prime}
      \}.
\end{equation*}
\end{theorem}
Next we complete Theorem \ref{th5.3alg} by establishing the
Neumann type identity, $\{A\}''= \Alg A$.  Note, that for the case
$p=2$ and $\alpha =1$ it follows from a general result of B.S.-Nagy and
C. Foias \cite{B. Sz.-Nagy and C. Foias 2} on a dissipative operator  with finite dimensional
imaginary part.

\begin{theorem}\label{NewmannforLp}
Suppose
      $
      A=\bigoplus_{i=1}^n \lambda_iJ^\alpha
      $
      is defined on
      $
      \bigoplus_{i=1}^{n}X
      $ and
      $\lambda_i$ satisfy condition \eqref{arg=arg}.
 Then
$\{A\}''= \Alg A$.
\end{theorem}
\begin{proof}
It is  known (and easily seen) that if $T_1$ and $T_2$ are bounded
operators on a Banach space $Y$, then  $\{T_1\oplus T_2\}''\subset
\{T_1\}''\oplus \{T_2\}''$. Hence $\{A\}''=\{\bigoplus_{i=1}^n
\lambda_iJ^\alpha\}''\subset \bigoplus_{i=1}^n\{
\lambda_iJ^\alpha\}''.$ \  It follows that any  $R\in \{A\}''$
admits a direct sum decomposition $R=\bigoplus_{i=1}^nR_i$ with
$R_i\in \{\lambda_iJ^\alpha\}'' = \{\lambda_iJ^\alpha\}', \  i\in
\{1,..., n\}$.  According to  \eqref{convoper}  $R_i$ admits a
representation  $(R_if)(x)= \frac{d}{dx}\int_0^xr_i(x-t)f(t)\,dt$,
where  $r_i\in L_{p'}[0,1]$ and it is such that $R_i\in [X]$.

Further, let  $K=(K_{ij})_{i,j=1}^n$ be an operator matrix  with
entries $K_{ij}= L_{a_{ij}}$ for $i>j$ and $K_{ij}=\mathbb O$ for
$i\leq j$. Let also
 $a_{ij}:=s_i^{-1}s_j$ for $i,j\in\{1,\dots,n\}$.
 Then, by Theorem \ref{pr6.3}, $K\in \{A\}'$. Clearly,  relation $RK=KR$
 yields
\begin{flalign}\label{2.11}
R_iL_{a_{i1}}=L_{a_{i1}}R_1, \qquad i\in\{2,\dots,n\}.
\end{flalign}
It is easily seen that
\begin{equation}\label{2.12}
(R_iL_{a_{i1}}f)(x)= \frac{d}{dx}\int_0^xr_i(x-t)f(s_i^{-1}t)\,dt,
\quad i\in\{2,\dots,n\}.
\end{equation}
 On the other hand,
\begin{align*}
(L_{a_{i1}}R_1 f)(x)
=&\frac{d}{dx_1}\int_0^{x_1}r_1(x_1-t)f(t)\,dt\Big\vert_{x_1=s_i^{-1}x}
\\
=s_i&\frac{d}{dx}\int_0^{s_i^{-1}x}r_1(s_i^{-1}x-t)f(t)\,dt
=\frac{d}{dx}\int_0^xr_1(s_i^{-1}(x-t))f(s_i^{-1}t)\,dt.
\end{align*}
Comparing  this relation with \eqref{2.12} and taking into account
\eqref{2.11} and the obvious relation  $\ran(L_{a_{i1}}) = X$, we
obtain $r_i(x)=r_1(s_i^{-1}x), \ i\in\{2,\dots,n\}.$
 By Theorem \ref{th5.3alg}, this means that $R\in \Alg A$, that is
 $\{A\}''\subset \Alg A$. Since the inclusion $\{A\}''\supset \Alg A$
is obvious, we get $\{A\}''= \Alg A$.
\end{proof}
In the following theorem we obtain  a description of $\Lat A$
similar to that of $\Lat T$ for $C_0$-contractions $T$ described
in Theorem \ref{Halmos}.  It is interesting to note that though a
description is completely the same, the operator  $A$ in not
accretive in $L_2[0,1]$ for $\alpha >1$ (cf. Remark
\ref{remarkHalmos} $(i)$).
\begin{theorem}\label{th5.3Halmos}
      Let
      $
      A=\bigoplus_{i=1}^n \lambda_iJ^\alpha
      $
      be defined on
      $
      \bigoplus_{i=1}^{n}X
      $
      and $\lambda_i$ satisfy conditions \eqref{arg=arg}.
       Then every invariant subspace of $A$ is the closure of the
        range $($ the kernel$)$  of a bounded linear transformation that commutes with
        $A$.
\end{theorem}
\begin{proof}
Alongside the operator $A$ we consider the operator
$A_1:=\bigoplus_{i=1}^n \lambda_1s_i^{-1}J.$ By Theorem
\ref{th5.3alg}, $\Alg A=\Alg (\bigoplus_{i=1}^n
\lambda_1s_i^{-\alpha} J^\alpha)= \Alg (\bigoplus_{i=1}^n
\lambda_1s_i^{-1}J)= \Alg A_1$ . Hence  $\Lat A=\Lat A_1$ and
$\{A\}'=\{A_1\}'$. So we can assume that $\lambda_1=1$ and
$\alpha=1$. We put
\begin{align*}
 K:=&\bigoplus_{i=1}^n
J\in \biggl[\bigoplus_{i=1}^nL_p[0,1],\bigoplus_{i=1}^nL_2[0,1]\biggr],\\
L:=&\bigoplus_{i=1}^n J\in
\biggl[\bigoplus_{i=1}^nL_2[0,1],\bigoplus_{i=1}^nL_p[0,1]\biggr],\\
B:=&\bigoplus_{i=1}^n s_iJ\in \biggl[\bigoplus_{i=1}^{n}L_2[0,1]\biggr].
\end{align*}
It is clear that $\ker K=\{0\}$, $\ker L=\{0\}$, $\overline{\ran
K}=\bigoplus_{i=1}^nL_2[0,1]$, $\overline{\ran
L}=\bigoplus_{i=1}^nL_p[0,1]$, $KA_1 = BK$ and $A_1L = LB$. Hence
$A_1$ is quasisimilar to $B$. So, we can assume that $A_1$ is
defined on $\bigoplus_{i=1}^nL_2[0,1]$. Note that  $A_1$ is
accretive,  since $s_i>0$ for $i\in \{1,...,n\}$.  Now the
assertions of the theorem follow from Theorem \ref{Halmos} (see
also Remark \ref{remarkHalmos} $(i)$).
\end{proof}
Next, we recall a description of $\Hyplat A$.
\begin{theorem}\label{theorem29} \cite[Proposition 4.8]{M.M.Malamud 2}
Suppose
      $
      A=\bigoplus_{i=1}^n \lambda_iJ^\alpha
      $
      is defined on
      $
      \bigoplus_{i=1}^{n}X
      $
      and $\lambda_i$ satisfy condition \eqref{arg=arg}.
      Then
  the lattice $\Hyplat A$ is of the form
      \begin{equation*}
             \Hyplat  A=
             \biggl\{\bigoplus_{i=1}^nE_{a_i}:\ (a_1,\ldots,a_n)
             \in P(s_1,\ldots,s_n)\biggr\},
      \end{equation*}
      where
      \begin{equation*}
      \begin{split}
             P(s_1,\ldots,s_n):=
             \bigl\{&(a_1,\ldots,a_n)\in [0,1]^n:\qquad\qquad \\
              &s_ia_{i+1}\leqslant s_{i+1}a_i
             \leqslant s_{i+1}-s_i+s_ia_{i+1},\
             1\leqslant i\leqslant n-1\bigr\}.
      \end{split}
      \end{equation*}
\end{theorem}
\begin{definition}$($ cf.
\cite{N.K.Nikolskii}$)$\label{definitioncyclic}
\begin{itemize}
      \item[(1)]  A subspace $E$ of a Banach space $X_1$ is called
      a cyclic subspace for an operator $T\in [X_1]$ if\ \
      $\myspan  \{T^n E :\ n \geqslant 0\} = X_1$;
      \item[(2)] a vector $f(\in X_1)$ is called cyclic for $T$ if\ \ $\myspan  \{T^nf:\ n\geqslant
       0\}=X_1$;
      \item[(3)] the set of all cyclic subspaces of an operator $T$
        is denoted by $\Cyc T$.
\end{itemize}
\end{definition}
\begin{definition}
\begin{itemize}
      \item[(1)] The number
      \begin{equation*}
            \mu_T:=
            \inf_E\{ \dim E:\ E
            \text{ is a cyclic subspace of the operator } T \text{ on } X_1\}
      \end{equation*}
      is called the  spectral\ multiplicity of
      an operator $T$ on $X_1$;\\
     \item[(2)] operator $T$ is called cyclic if $\mu_T=1$.
\end{itemize}
\end{definition}
It is well known that the concept of spectral multiplicity plays an
important role in control theory (see for instance \cite{Wonham}).
Investigating some other problems of control theory, N.K.
Nikol'skii and V.I. Vasjunin \cite{Nikolskii and Vasjunin}
introduced one more  "cyclic" characteristic of an operator.
\begin{definition}\cite{Nikolskii and Vasjunin}
       Let $T\in[X]$. Then
      $$
      \disc T:=\sup\limits_{E\in\Cyc T}\min\{\dim E^{\prime}:\ E^{\prime}\subset E,E^{\prime}\in\Cyc T\}.
      $$
       $\disc T$ is called a disc-characteristic of an operator
       $T$. $($"disc" is the abbreviation of "Dimension of the Input Subspace of
       Control".$)$
\end{definition}
Clearly,  $\disc T\geqslant \mu_T$.

To present a description of $\Cyc A$ we recall  the following
definition.
      \begin{definition} \cite{M.M.Malamud 1, M.M.Malamud 2, M.S.Nikol'skii})
\label{Def5.2}
      The determinant of a functional matrix
      $
      F(x)=(f_{ij}(x))_{i,j=1}^n
      $
      $(f_{ij}\in X)$
      calculated with respect to the convolution product
      \begin{equation*}
             (f*g)(x) =
             \int_0^x f(x-t)g(t)\,dt =
             \int_0^x g(x-t)f(t)\,dt = (g*f)(x)
      \end{equation*}
      is called $*$ - determinant
      and is denoted by $*-\det F(x)$.
          Similarly, $*$ - minors of $F(x)$ are the minors calculated
          with respect to the convolution product.
      $\srank F(x)$ will is the highest order of $*$-minors of $F(x)$
      satisfying $\varepsilon$-condition \eqref{equ5.1}.
\end{definition}
Next we complete \cite[Theorem 2.3]{M.M.Malamud 2} by computing $\disc A$.
\begin{theorem}\label{th5.3}
      Suppose
      $
      A=\bigoplus_{i=1}^n \lambda_iJ^\alpha
      $
      is defined on
      $
      \bigoplus_{i=1}^{n}X
      $
      and $\lambda_i$ satisfy condition \eqref{arg=arg}.
      Then the system $\{f_l\}_{l=1}^N$ of vectors
      \begin{flalign*}
&\ & f_l=f_{l1}\oplus\dots\oplus f_{ln}\in
             \bigoplus_{i=1}^{n}X,
             &\ &
            l\in\{1,\dots,N\},\ \ i\in\{1,\dots,n\}
      \end{flalign*}
      generates a  cyclic subspace  for the operator $A$
      if and only if
\begin{itemize}
      \item[(i)]
      $
      N\geqslant n
      $;
      \item[(ii)]
      the matrix
      \begin{equation*}
             F_n(x)=
             \left(
                \begin{matrix}
                f_{11}(s_1x)&f_{12}(s_2x)&\ldots&f_{1n}(s_nx)\\
                \vdots&\vdots&&\vdots                  \\
                f_{N1}(s_1x)&f_{N2}(s_2x)&\ldots&f_{Nn}(s_nx)
               \end{matrix}
             \right)
      \end{equation*}
      is of maximal $\srank$ , namely, $\srank F_n(x)=n$;
      \item[(iii)]
      $\disc A=\mu_A=n$.
\end{itemize}
\end{theorem}
\begin{proof}
$(i)$, $(ii)$ and the equality $\mu_A=n$ were proved in \cite[Theorem
2.3]{M.M.Malamud 2}(see also  \cite[Proposition 3.2]{I.Yu.
Domanov0} for another proof).

$(iii)$
 Let us prove that $\disc A=n$. Let
$E=\myspan\{f_1,\dots,f_N\}$ be an $N$-dimensional subspace cyclic
for the operator $A$. It is necessary to show that this space
contains an $n$-dimensional subspace which is also cyclic for the
operator $A$. Since $\srank F_n(x)=n$, it follows that there
exists an $n\times n$ submatrix $G_n(x)$ of $F_n(x)$ such that
$\srank G_n(x)=n$. Hence we can choose $n$- vectors
$f_{i_1},\dots, f_{i_n}$ $(i_1,\dots,i_n\in\{1,\dots, N\})$ such
that $\myspan\{f_{i_1},\dots, f_{i_n}\}$ is a cyclic subspace for
$A$.
\end{proof}
\begin{corollary}\label{cor216}
 Let $K\in\{J^\alpha\}'$ and $K_n=\bigoplus_{i=1}^n K$ be defined on $\bigoplus_{i=1}^n L_2[0,1]$.
Then $\mu_{K_n}\geq n$.
\end{corollary}
\begin{proof}
 It follows from Theorem \ref{th5.3alg} that $K_n\in \Alg A$, where $A=\bigoplus_{i=1}^n J$ is defined $\bigoplus_{i=1}^n L_2[0,1]$. Hence, by Theorem  \ref{th5.3} $\mu_{K_n}\geq \mu_A=n$.
\end{proof}
\begin{remark}
In the recent paper \cite[Proposition 7.6]{Bermudo and Rodriguez Shkarin} Corollary \ref{cor216}
was proved for the case $n=2$.
\end{remark}
Next we recall the following notation.
       Let $T_j\in[X_j]$ $(j=1,2)$  and
       $R\in \Cyc (T_1\oplus T_2)$.
       It is clear that $P_jR\in \Cyc T_j$, where
       $P_j$ is the  projection from  $X_1\oplus X_2$ onto $X_j$, $j\in\{1,2\}$.
       Following \cite{Nikolskii and Vasjunin}, we write
       $$
          \Cyc (T_1\oplus T_2)=\Cyc  T_1\vee \Cyc T_2
       $$
       if $P_jR\in \Cyc  T_j$ $(j=1,2)$ yields $R\in \Cyc (T_1\oplus T_2)$
       for every $R\subset X_1\oplus X_2$.
       In particular, if  $\Lat (T_1\oplus T_2)=\Lat  T_1\oplus \Lat  T_2$
       then $\Cyc (T_1\oplus T_2)=\Cyc  T_1\vee \Cyc T_2$.

Next we complete  \cite[Proposition 4.1, Proposition 4.2]{M.M.Malamud 2}.
\begin{theorem}\label{split}
      Suppose
      $
      A=\bigoplus_{j=1}^r \lambda_jJ^\alpha
      $
      is defined on
      $
      \bigoplus_{j=1}^rX
      $
      and
\begin{equation}
 \arg\lambda_{i}\ne \arg\lambda_{j}\pmod{2\pi},\qquad
 1\leqslant i<j\leqslant r.\label{algnealg}
\end{equation}
      Then
            \begin{align}
            &\Alg A=\{A\}'= \{A\}''=\bigoplus_{j=1}^r\Alg J^\alpha=\bigoplus_{j=1}^r\{J^\alpha\}'=\bigoplus_{j=1}^r\{J^\alpha\}'',
              \\
             &\Lat  A=\Hyplat  A
             =\bigoplus_{j=1}^r \Lat J^\alpha=\bigoplus_{j=1}^r \Hyplat J^\alpha,
             \label{eqref4.1bb}
             \\
             &\Cyc  A
             =\bigvee_{j=1}^r \Cyc J^\alpha ,\label{neweq3}\\
             &\disc A=\mu_A=1.
             \label{eqref4.1dd}
      \end{align}
\end{theorem}
\begin{proof}
\eqref{eqref4.1bb}-\eqref{eqref4.1dd} and the splitting of
$\{A\}'$ and $\{A\}''$ were proved in \cite{M.M.Malamud 1}, \cite{M.M.Malamud 2}. We
present two different proofs of the splitting of $\Alg A$ due to the first
and to the second  author, respectively.

\emph{First  proof.}
 We will derive the splitting of $\Alg A$ from the splitting of $\Cyc
A$.

By \eqref{neweq3} $g:= \frac{x^{\alpha-1}}{\Gamma(\alpha)}\oplus\dots\oplus
\frac{x^{\alpha-1}}{\Gamma(\alpha)}\in \Cyc A$. Hence,
there exists a sequence $\{P_n(x)\}_{n=1}^\infty$ such that
 $\slim\limits_{n\rightarrow\infty} P_n(A)g
      =0\oplus\dots\oplus 0\oplus\frac{x^{\alpha-1}}{\Gamma(\alpha)}$.
       We claim that
\begin{equation}
        \slim\limits_{n\rightarrow\infty}AP_n(A)=
       \mathbb O\oplus\dots\oplus\mathbb O\oplus
       \lambda_rJ^\alpha.\label{splitpol}
\end{equation}
 Indeed, for any $f=f_1\oplus\dots\oplus f_r\in\bigoplus_{j=1}^rX$ one has
\begin{align*}
&\slim\limits_{n\rightarrow\infty}
                       AP_n(A)f=
      \slim\limits_{n\rightarrow\infty}
\left( \lambda_1J^\alpha P_n(\lambda_1J^\alpha)f_1\oplus\dots\oplus
\lambda_rJ^\alpha P_n(\lambda_rJ^\alpha)f_r\right) \\
     &=\slim\limits_{n\rightarrow\infty}
\left(\lambda_1\frac{x^{\alpha-1}}{\Gamma(\alpha)}*(P_n(\lambda_1J^\alpha)f_1)(x)\oplus\dots\oplus
\lambda_r\frac{x^{\alpha-1}}{\Gamma(\alpha)}*(P_n(\lambda_rJ^\alpha)f_r)(x)\right)
\\
&=\slim\limits_{n\rightarrow\infty}\left(
  \lambda_1f_1*(P_n(\lambda_1J^\alpha)\frac{x^{\alpha-1}}{\Gamma(\alpha)})\oplus\dots\oplus
,\lambda_rf_r*(P_n(\lambda_rJ^\alpha)\frac{x^{\alpha-1}}{\Gamma(\alpha)}\right)
\\
&=\lambda_1f_1*0\oplus\lambda_2f_2*0\oplus\dots\oplus\lambda_rf_r*\frac{x^{\alpha-1}}{\Gamma(\alpha)}
=\diag(\mathbb O,\dots,\mathbb O,\lambda_rJ^\alpha)f.
\end{align*}
So \eqref{splitpol} is proved. A similar argument shows that for
any $j\in\{1,\dots,r\}$ there exists a sequence of polynomials
$\{P_{j,n}\}_{n=1}^\infty$ such that
\begin{equation*}
        \slim\limits_{n\rightarrow\infty}AP_{j,n}(A)=
      \mathbb O\oplus\dots\oplus
        \mathbb O\oplus\lambda_jJ^\alpha\oplus\mathbb O\oplus\dots\oplus\mathbb
        O.
\end{equation*}
Hence the splitting of $\Alg A$ is proved.

\emph{Second proof.}  Keeping in mind notations of Theorem
\ref{th4.1Lp} (see below), for any $j\in\{1,\dots,r\}$ we let
$n_j=n$ and  $\lambda_{j1}:=\dots:=\lambda_{jn}:=\lambda_j$. Then
setting  $A(j) := \bigoplus_{i=1}^{n_j}\lambda_{ji} J^\alpha$ we
rewrite $A(j)$ and $A$ as
$$
A(j) = \bigoplus_{i=1}^{n} \lambda_j J^\alpha = (\lambda_j
J^\alpha)^{(n)} \quad  \text{and} \quad
A=\bigoplus_{j=1}^rA(j)=\bigoplus_{j=1}^r(\lambda_j J^\alpha)^{(n)},
$$
where the factors $\lambda_j$ have different arguments,
$\lambda_j\not = \lambda_k$ for $j\not = k.$ Therefore by Theorem
\ref{th4.1Lp}  the lattice $\Lat (\bigoplus_{j=1}^r(\lambda_j
J^\alpha)^{(n)})$ splits,  $\Lat (\bigoplus_{j=1}^r(\lambda_j
J^\alpha)^{(n)})=\bigoplus_{j=1}^r\Lat (\lambda_j J^\alpha)^{(n)}$.
One completes the proof by applying  Theorem \ref{corRadjRos} with
$T_j=\lambda_j J^\alpha,\  j\in\{1,\dots,n\}$.
\end{proof}
\begin{remark}\label{OsilenkerShulman}
Some particular statements of Theorem \ref{split} were obtained in
 \cite{A.Atzmon, L.T.Hill, B.P. Osilenker and V.S.Shulman, B.P. Osilenker and
 V.S.Shulman2} for the case $p=2$.

 Namely,  A. Atzmon \cite{A.Atzmon} proved that for every integer $k\geq
 2$, the operator $iJ^{1-1/k}\oplus e^{\frac{\pi i}{2k}}J^{1-1/k}$ is
 cyclic.

In \cite{B.P. Osilenker and V.S.Shulman, B.P. Osilenker and
V.S.Shulman2} B.P. Osilenker and V.S. Shulman proved  that
\eqref{algnealg} implies the splitting of  $\Lat
(\bigoplus_{j=1}^r\lambda_j J).$ Their proof cannot be extended to
the case $\alpha \not =1.$

L.T. Hill \cite{L.T.Hill}  showed that if $\alpha\in (0,1)$ and
$\lambda$ is a nonzero complex number, then $\Lat (J^\alpha\oplus
\lambda J^\alpha)$ splits if and only if $\lambda$ is not
positive. His proof cannot be extended  neither to the case
 of $\alpha >1$ nor to the number of summands $n>2.$

\end{remark}
The following result is easily implied by combining Theorems
\ref{pr6.3} and \ref{split}.
\begin{corollary}\label{mainfor01}\cite{M.M.Malamud 4},\cite{M.M.Malamud 2}
     Let $c\in\mathbb C$ and let
     $R\in [X]$ be a solution
     of the equation
     $
      R J^\alpha=c J^\alpha R
     $.
     Then the following statements hold
\begin{itemize}
     \item[(i)] if $c\not\in\mathbb R_+$, then $R=\mathbb O$;
     \item[(ii)] if $c=a^\alpha>0,\ a>0$, then
     $
     R\in L_a\{J^\alpha\}^{\prime}$, where $L_a$ is defined by \eqref{neweq2}.
\end{itemize}
\end{corollary}
\begin{remark}\label{remee}
\begin{itemize}
\item[(i)] It was shown in \cite{Ciprian Foias J. P. Williams} that the operators
 $J$ and $cJ$ are similar if and only if $c=1$.   Corollary \ref{mainfor01}  implies that  operators
 $J^\alpha$ and $cJ^\alpha$ are not even quasisimilar for any $c\ne 1$.
\item[(ii)] In particular cases Corollary \ref{mainfor01} $(i)$ was recently reproved by another method in
\cite{A.Biswas and A.Lambert and S.Petrovic}, \cite{M.T.Karaev} (the case $\alpha =1$, $p=2$) and in
\cite{A.Biswas and S.Petrovic} (the case $\alpha\in\mathbb Z_+\setminus\{0\}$, $p=2$).
Some solutions $R$ of the equation $R J^\alpha=c J^\alpha R$ in the case $c>0$, $\alpha\in \mathbb Z_+$ were also indicated in
\cite{A.Biswas and A.Lambert and S.Petrovic}, \cite{A.Biswas and S.Petrovic}, \cite{M.T.Karaev}.
\end{itemize}
\end{remark}
We need the following lemma in the sequel.
\begin{lemma}\label{prodisc}
Suppose that $A\in [X_1]$ is quasisimilar to $B\in [X_2]$ with
intertwining deformations  $L$  and $K$. That is,  $AL=LB$ and
$KA=BK$. Let also $LK=A^2$ and $KL=B^2$. Then
\begin{itemize}

\item[(i)] $E\in \Cyc A\Leftrightarrow \overline{KE}\in \Cyc B $;

\item[(ii)] $F\in \Cyc B\Leftrightarrow \overline{LF}\in \Cyc A $;

\item[(iii)] $\disc A=\disc B$.
\end{itemize}
\begin{proof}
The proof is left for the reader.
\end{proof}
\end{lemma}
Now we can consider the case of any diagonal nonsingular
matrix $B$.

Next we complete \cite[Proposition 3.2, Theorem 3.4, Corollary 3.5, Theorem 4.10, Theorem 4.11]{M.M.Malamud 2}.
\begin{theorem}\label{th4.1Lp}
     Suppose
     $
     A(j):= \bigoplus_{i=1}^{n_j}\lambda_{ji} J^\alpha
     $
     is defined on $
     \bigoplus_{i=1}^{n_j}X
     $,
      $
     \ j\in\{1,\dots,r\}
     $
     and
     $
     A:=\bigoplus_{j=1}^r A(j)
     $
     is defined on
     $
     \bigoplus_{j=1}^r(\bigoplus_{i=1}^{n_j}X)
     $.
Let also
     \begin{align*}
 \arg\lambda_{j1}&=\arg\lambda_{ji}\pmod{2\pi}, &
     &1\leqslant j \leqslant r,\quad 1\leqslant i \leqslant n_j,\\
  \arg\lambda_{i1}&\ne \arg\lambda_{j1}\pmod{2\pi}, &
&1\leqslant i<j\leqslant r.
     \end{align*}
    Then
      \begin{align}
             \Alg  A
             &=\bigoplus_{j=1}^r \Alg  A(j),
            \label{eqref4.1a}\\
      \{A\}'&= \bigoplus_{j=1}^r\{A(j)\}',
      \label{eqref4.1e}\\
            \{A\}''&= \bigoplus_{j=1}^r\{A(j)\}'',
            \label{eqref4.1f}\\
      \Lat  A
      &=\bigoplus_{j=1}^r \Lat  A(j),
      \label{eqref4.1b}\\
             \Hyplat  A
             &=\bigoplus_{j=1}^r \Hyplat  A(j),
             \label{eqref4.1c}\\
       \Cyc  A
        &=\bigvee_{j=1}^r \Cyc A(j),
       \label{eqref4.1d}\\
              \disc A &= \mu_A =\max\limits_{1\leqslant j\leqslant
              r}\mu_{A(j)}.\nonumber
      \end{align}
\end{theorem}
\begin{proof}
Relations \eqref{eqref4.1e}-\eqref{eqref4.1d} and the equality
$\mu_A =\max\limits_{1\leqslant j\leqslant r}\mu_{A(j)}$ were
proved in \cite{M.M.Malamud 2}.
 Let us prove \eqref{eqref4.1a}. By Theorem \ref{split},
 for any $j\in \{1,\dots,r\}$
$$
\mathbb O\oplus\dots\oplus \mathbb O\oplus
\lambda_{j1}J^\alpha\oplus \mathbb O\oplus\dots\oplus \mathbb O\in
\Alg(\lambda_{11}J^\alpha\oplus\dots\oplus\lambda_{r1}J^\alpha).
$$
Thus,  by Theorem \ref{th5.3alg} we have that
$$
\mathbb O\oplus\dots\oplus \mathbb O\oplus A(j)\oplus \mathbb
O\oplus\dots\oplus \mathbb O\in \Alg A,
$$
and hence \eqref{eqref4.1a} is proved.

Let us prove that $\disc A = \mu_A$. Assume that $A_1:=A$ is
defined on\\ $\bigoplus_{j=1}^r(\bigoplus_{i=1}^{n_j}L_2[0,1])$. Then
\cite[Statement 1.13]{Nikolskii and Vasjunin2}  and \cite[Corollary
13]{Nikolskii and Vasjunin} imply the equality $\disc A_1 =
\mu_{A_1}$. We define
\begin{align*}
K&:=\bigoplus_{j=1}^r\bigoplus_{i=1}^{n_j}\lambda_{ji}J^\alpha\in
\biggl[\bigoplus_{j=1}^r\Bigl(\bigoplus_{i=1}^{n_j}L_2[0,1]\Bigr),\bigoplus_{j=1}^r\Bigl(\bigoplus_{i=1}^{n_j}L_p[0,1]\Bigr)\biggr],\\
L&:=\bigoplus_{j=1}^r\bigoplus_{i=1}^{n_j}\lambda_{ji}J^\alpha\in
\biggl[\bigoplus_{j=1}^r\Bigl(\bigoplus_{i=1}^{n_j}L_p[0,1]\Bigr),\bigoplus_{j=1}^r\Bigl(\bigoplus_{i=1}^{n_j}L_2[0,1]\Bigr)\biggr],
\end{align*}
and $A_2:=A$. It is clear that $K$ and $L$ are deformations and
$A_1L=LA_2$, $KA_1=A_2K$. Now  application of Lemma \ref{prodisc}
completes the proof.
\end{proof}
\section
{The operator $A_{k,0}$}
 \label{OperatorAk0}
Let  $J_{k,l}^\alpha$  stand for the operator   $J_k^\alpha$
acting on the subspace $E_l^k$ of  $W_p^k[0,1]$ defined by
\eqref{equ1.7} $(l\leqslant k-1)$ and $E_k^k := W_p^k[0,1]$.

Next we establish isometric equivalence of $J_{k,0}^\alpha$ and $J^\alpha$.
\begin{lemma}\label{scalL2.1}
      The operator $J_{k,l}^\alpha$ defined on $E_l^k$ is
      isometrically equivalent to the operator $J_l^\alpha$
      defined on $W_p^l[0,1]$.
        In particular, the operator $J_{k,0}^\alpha$ defined on
        $W_{p,0}^k[0,1]$ is isometrically equivalent to the operator
        $J_{0}^\alpha=:J^\alpha$ defined on $W_p^0[0,1]=L_p[0,1]$.
\end{lemma}
\begin{proof}
It is clear that the operator $ U=\frac{d^{k-l}}{dx^{k-l}}:\
E_l^k\rightarrow W_p^l[0,1] $ isometrically maps $E_l^k$ on
$W_p^l[0,1]$.
  Moreover,
  \begin{equation*}
      U^{-1}=U^*=
      J^{k-l} :\ W_p^l[0,1]\rightarrow E_l^k.
  \end{equation*}
The assertion follows now from the identity $
J_{k,l}^\alpha=U^{-1}J_l^\alpha U$.
\end{proof}
\begin{corollary}\label{corollary32}
The operator
$A_{k,0}:=\bigoplus_{i=1}^n\lambda_i J_{k_i,0}^\alpha$
defined on $\bigoplus_{i=1}^n W_{p,0}^{k_i}[0,1]$
is isometrically equivalent to the
operator
 $A:=\bigoplus_{i=1}^n\lambda_i J^\alpha$ defined on
$\bigoplus_{i=1}^n L_p[0,1]$.
\end{corollary}
Corollary \ref{corollary32} makes it possible to translate all results on the operator
 $A$ defined on $\bigoplus_{i=1}^n L_p[0,1]$ to the results on operator $A_{k,0}$ defined on
  $\bigoplus_{i=1}^n W_{p,0}^{k_i}[0,1]$. For instance, Theorem \ref{th5.3Halmos} takes the following form
\begin{theorem}\label{th5.3Halmosinsobolev}
      Let
      $
      A_{k,0}:=\bigoplus_{i=1}^n\lambda_i J_{k_i,0}^\alpha
      $
      be defined on
      $
      \bigoplus_{i=1}^n W_{p,0}^{k_i}[0,1]
      $
      and $\lambda_i$ satisfy condition \eqref{arg=arg}.
       Then every invariant subspace of $A_{k,0}$ is the closure of the
        range $($ the kernel$)$  of a bounded linear transformation that commutes with
        $A_{k,0}$.
\end{theorem}
\section
{The operator $A_k$}\label{OperatorAk}
This section contains the main results of the paper. Namely, we
described  the spectral properties  of the operator
$A_k:=\bigoplus_{j=1}^n\lambda_j J_{k}^\alpha$ defined on $X^{(n)}=
\bigoplus_1^nX $ where  $X = W_p^{k}[0,1]$.
\subsection {The algebra $\Alg A_k$}\label{sec7}
\begin{theorem}\label{AlgProp2}
Suppose
      $
      A_k=\bigoplus_{i=1}^n\lambda_i J_k^\alpha
      $
      is defined on
      $
      \bigoplus_{i=1}^n W_p^k[0,1]
      $
      and
\begin{equation}
\lambda_i=
      \lambda_1/s_i^{\alpha},\ \ \ 1=
      s_1\leqslant s_2\leqslant \ldots\leqslant s_n,\qquad
 i\in\{1,\dots,n\}.\label{=arg}
\end{equation}
Let also
\begin{equation}
  R:=\bigoplus_{i=1}^n R_i \in \biggl[\bigoplus_{i=1}^n
W_p^k[0,1]\biggr],\quad
 \ (R_if)(\cdot)=c_if(\cdot)+(r_i*f)(\cdot),\ \ i\in\{1,\dots,n\}.\label{neweq5}
\end{equation}
Then the following is true:
\begin{itemize}
      \item[(1)]
      if  $1\leqslant\alpha\leqslant k-1$, then
      \begin{equation}\label{algargarg1}
      \begin{split}
             \Alg  A_k=
             \bigl\{&R :\ c_1=\dots=c_n\in \mathbb C;\  r_1\in
             W_p^{k-1}[0,1];\
              r_i(x)=s_i^{-1} r_1(s_i^{-1}x), \\  &1\leqslant i \leqslant
              n;
             \ r_1^{(l)}(0)=0,\ \ l\ne m\alpha-1,\ \
             \ 1\leqslant m\leqslant [(k-1)/{\alpha}]\bigr\};
      \end{split}
      \end{equation}
      \item[(2)]
      if $2\leqslant k\leqslant \alpha+\frac{1}{p}$, then
      \begin{equation}\label{algargatg2}
         \begin{split}
             \Alg  A_k=
             \bigl\{R:\ c_1=\dots=c_n\in \mathbb C;
             \ r_1\in W_{p,0}^{k-1}[0,1]&;\\  r_i(x):=s_i^{-1} r_1(s_i^{-1}x),\ \  &\ 1\leqslant i \leqslant
              n\bigr\}.
         \end{split}
      \end{equation}
\end{itemize}
\end{theorem}
\begin{proof}
Let
\begin{equation*}
(M_if)(x): = f(s_i^{-1}x),\quad (N_if)(x):=\begin{cases}
f(s_ix),& x\in [0,s_i^{-1}],\\
\sum\limits_{m=0}^{k-1}\frac{(xs_i-1)^m}{m!}f^{(m)}(1),&x\in
[s_i^{-1},1].
\end{cases}
\end{equation*}
It can  easily be checked that
\begin{equation*}
(\lambda_iJ_k^\alpha)^m=M_i(\lambda_1J_k^\alpha)^m N_i, \qquad
m\in{\mathbb Z}_+, \qquad  i\in\{1,\dots,n\}.
\end{equation*}
Setting $A_i:= \lambda_iJ_k^\alpha$ and applying Lemma
\ref{alglemma} we obtain
\begin{equation}\label{4.5}
       \Alg A=\Bigl\{R=\bigoplus_{i=1}^n R_i: R_1\in \Alg
      (\lambda_1J_k^\alpha),\ \ R_i = M_iR_1N_i,\ \
      i\in\{2,\dots,n\}\Bigr\}.
\end{equation}
Next we confine ourselves to the case $1\leqslant\alpha\leqslant
k-1$.   The case $2\leqslant k\leqslant \alpha+\frac{1}{p}$ is
considered similarly. By \eqref{algsobolev}, $R_1\in \Alg
      (\lambda_1J_k^\alpha)$ if and only if
\begin{equation}\label{convoperrsob}
\begin{split}
R_1\ :\ f(x)\rightarrow c_1f(x)+\int_0^xr_1(x-t)f(t)\,dt,\quad \
c_1\in\mathbb C,\quad r_1\in W_p^{k-1}[0,1],
 \\
 r_1^{(l)}(0)=0,\ \ l\ne m\alpha-1,\ \
             \ 1\leqslant m\leqslant [(k-1)/{\alpha}].
\end{split}
\end{equation}
Straightforward calculations show that
$$
(M_iR_1N_if)(x)=c_1+\int_0^xs_i^{-1}r_1(s_i^{-1}(x-t))f(t)\,dt,\qquad
i\in\{2,\dots,n\}.
$$
Combining the last relations with \eqref{4.5} we arrive at the
required description.
        \end{proof}
In the proof of the following theorem we need a concept of the
weak operator topology in the algebra $[X]$. Recall the following
definition.
     \begin{definition}\label{defWeakTop}
Let $\{f_i\}_{i=1}^N)$ and $\{g_i\}_{i=1}^N$ be the sets of unit
vectors in $X$ and $X^*$, respectively, and let $\varepsilon$ be a
positive number. For any $R\in B[X]$ define $\mathcal V :=
\mathcal V (\varepsilon;\{f_i,g_i\}_{i=1}^N)$ to be the set of all
operators $T$ satisfying
$$
|(T - R)f_i, g_i)|<\varepsilon, \qquad i\in\{1,\dots,N\}.
$$
Then $\mathcal V$ is a weak  neighborhood of $R$ and the family of
all such sets $\mathcal V$ is a base of weak neighborhoods of $R.$
           \end{definition}
\begin{theorem}\label{AlgProp1}
 Suppose
      $
      A_k=\bigoplus_{j=1}^r\lambda_j J_k^\alpha
      $
      is defined on
      $
      \bigoplus_{j=1}^r W_p^k[0,1]
      $
      and
\begin{equation*}
\arg\lambda_i\ne \arg\lambda_j\pmod{2\pi},\qquad
1\leqslant i<j\leqslant r.
\end{equation*}
Let also
\begin{equation*}
        R:=\bigoplus_{j=1}^r R_j \in \biggl[\bigoplus_{j=1}^r
W_p^k[0,1]\biggr],\quad
 \ (R_jf)(\cdot)=c_jf(\cdot)+(r_j*f)(\cdot), \quad   j\in\{1,\dots,r\}.
\end{equation*}
Then the following are true:
\begin{itemize}
\item[(1)]
      if\ \  $1\leqslant \alpha\leqslant k-1$, then
      \begin{equation}
      \begin{split}
      \Alg  A_k=\Bigl\{&R:\ c_1=\dots =c_r\in  \mathbb C;\
      r_j\in  W_p^{k-1}[0,1],\\
      &r_j^{(\alpha m-1)}(0)=
      (\lambda_j\lambda_1^{-1})^{m}r_1^{(\alpha m-1)}(0),
      \ \ m\leqslant \Bigl[\frac{k-1}{\alpha}\Bigr],\ \ 1\leqslant j\leqslant r
      ;\\
       &r_j^{(l)}(0)=0,\ \ l\ne \alpha m-1,\
             \ m\leqslant [(k-1)/{\alpha}],\ 1\leqslant j\leqslant r
             \Bigr\};\label{Alg1}
      \end{split}
      \end{equation}
\item[(2)]
      if\ \  $2\leqslant k\leqslant \alpha+\frac 1p$, then
      \begin{equation*}
             \Alg  A_k=\bigl\{R:\ c_1=\dots =c_r\in \mathbb C;\ \
             r_j \in W_{p,0}^{k-1}[0,1],\ \ \ 1\leqslant j\leqslant r \bigr\}.
      \end{equation*}
\end{itemize}
\end{theorem}
\begin{proof}
$(i)$ Theorem \ref{split}  and Corollary \ref{corollary32} imply
that $\mathbb O\oplus\dots\oplus
        \mathbb O\oplus\lambda_jJ_{k,0}^\alpha\oplus\mathbb O\oplus\dots\oplus\mathbb
        O\in \Alg (\bigoplus_{j=1}^r\lambda_j
        J_{k,0}^\alpha)$
for any $j\in\{1,\dots,r\}$. It easily implies that $M_j:=\mathbb
O\oplus\dots\oplus
        \mathbb O\oplus\ (\lambda_jJ_k^\alpha)^{k+1}\oplus\mathbb O\oplus\dots\oplus\mathbb
        O\in A^k\Alg A_k$.
 Thus $M_j\in \Alg A_k$ and
\eqref{algsobolev} implies that if
 either $\alpha\in \mathbb Z_+\setminus\{0\}$ or $\alpha
>k-\frac{1}{p}$, then
\begin{equation}
             \Alg  A_k\supset\bigl\{R:\ c_1=\dots =c_r\in \mathbb C;\
             r_j \in W_{p,0}^{k-1}[0,1], \ \ 1\leqslant j\leqslant r \bigr\}.
	     \label{Alg2inclusion}
\end{equation}

$(ii)$ Let  $2\leqslant k\leqslant \alpha+\frac 1p$. Then combining the obvious
inclusion
 $\Alg A_k\subset \bigoplus_{j=1}^r \Alg \lambda_j J_k^\alpha$ with \eqref{algsobolev}
 we arrive at opposite inclusion in \eqref{Alg2inclusion}. Thus, (2)
is proved.

$(iii)$ Let us prove the inclusion "$\subset$" in \eqref{Alg1}.
Description \eqref{algsobolev} and inclusion
 $\Alg A_k\subset \bigoplus_{j=1}^r \Alg \lambda_j J_k^\alpha$ imply
 that
\[
\begin{split}
  \Alg  A_k\subset\bigl\{&R:\ c_j\in  \mathbb C;\
      r_j\in  W_p^{k-1}[0,1],\\
   &r_j^{(l)}(0)=0,\ \ l\ne \alpha m-1,\
             \ m\leqslant [(k-1)/{\alpha}],\ 1\leqslant j\leqslant r
             \bigr\}.
\end{split}
\]

For $j\in\{1,\dots,r\}$ and $m\in\{1,\dots,[\frac{k-1}{\alpha}]\}$
by definition, put :
$$
x_{jm}:=\underbrace{0\oplus\dots\oplus 0\oplus \one}_j\oplus\, 0\oplus\dots
\oplus 0,\quad
 y_{jm}:=\underbrace{0\oplus\dots\oplus 0\oplus
\frac{x^{\alpha m}}{\Gamma(\alpha m)}}_j\oplus\, 0\oplus\dots \oplus 0.
$$
Let $R:=\bigoplus_{j=1}^rR_j\in\Alg A_k$. Choose $\varepsilon_1>0$
and  put
$$
\varepsilon:=\min\bigl\{ |{2^{-1}\lambda_j^m\varepsilon_1}| :\  1\leq
j\leq r,\ \ 0\leq m\leq k_1 \bigr\} \qquad \text {and} \qquad  k_1
:=\Bigl[\frac{k-1}{\alpha}\Bigr].
$$
Next, choose  vectors $\{x_{jm}\}_{j,m=1}^{r,k_1}$ and $\{
y_{jm}\} _{j,m=1} ^{r, k_1}$ belonging to $W_p^k[0,1]$ and\\
$(W_p^k[0,1])^*=W_{p'}^k[0,1]$, respectively and define a weak
neighborhood \\ $\mathcal V := \mathcal V (\varepsilon;
\{x_{jm}\}_{j,m=1}^{r,k_1}, \{ y_{jm}\} _{j,m=1} ^{r, k_1})$ of
$R$ according to Definition \ref{defWeakTop}. Then by definition
of $\Alg  A_k$ there exists a polynomial
$p(x):=\sum_{l=0}^Na_lx^l$ such that $p(A_k)$ belongs to the weak
neighborhood $\mathcal V$ of $R$, $p(A_k)\in \mathcal V$, that is
\begin{equation}\label{eq1}
|((R-p(A_k))x_{jm},y_{jm})|<\varepsilon,\qquad
j\in\{1,\dots,r\},\qquad m\in\{0,\dots, k_1\}.
   \end{equation}
It is clear that \eqref{eq1} is equivalent to the following system
\begin{equation*}
\Big|\Bigl(\bigl(R_j-p(\lambda_jJ_k^\alpha)\bigr)\one,\frac{x^{\alpha
m}}{\Gamma(\alpha m)}\Bigr)\Big|<\varepsilon,\quad j\in\{1,\dots,r\},\quad
m\in\{0,\dots,k_1\}.
\end{equation*}
After simple computations this systems reduces to the following
one
$$
|c_j-a_0|<\varepsilon,\quad \bigg|\frac{r_j^{(\alpha
m-1)}(0)}{\lambda_j^m}-a_m\bigg|<\frac{\varepsilon}{\lambda_j^m},\quad
j\in\{1,\dots,r\},\quad m\in\{0,\dots,k_1\}.
$$
Finally, triangle inequality implies that
$$
|c_1-c_j|<2\varepsilon\leq \varepsilon_1,\quad \bigg|\frac{r_1^{(\alpha
m-1)}(0)}{\lambda_1^m}-\frac{r_j^{(\alpha
m-1)}(0)}{\lambda_i^m}\bigg|<\frac{2\varepsilon}{\lambda_j^m}\leq\varepsilon_1,\quad
m\in\{0,\dots,k_1\}.
$$
Hence,
\begin{equation*}
\begin{split}
&c_j=c_1,\\
 &r_j^{(\alpha m-1)}(0)=
      (\lambda_j\lambda_1^{-1})^{m}r_1^{(\alpha
      m-1)}(0), \quad
   m\in\{1,\dots,k_1\},  \quad j\in\{1,\dots,r\}.
\end{split}
\end{equation*}
Thus, the inclusion "$\subset$" in \eqref{Alg1} is proved.

$(iiii)$ Let $R$ belongs to the algebra defined by the right side
of \eqref{Alg1}. Since $r_j\in  W_p^{k-1}[0,1]$, it follows that
\begin{flalign*}
 r_j(x) = r_{j,0}+r_{j,k-2}:= \left(r_j(x)-
\sum\limits_{i=0}^{k-2}r_j^{(i)}(0)\frac{x^i}{i!}\right)+
\sum\limits_{i=0}^{k-2}r_j^{(i)}(0)\frac{x^i}{i!},\ \ j\in\{1,\dots,r\}.
\end{flalign*}
According to this decomposition we can write $R=R_0+R_{k-2}$,
where
$$
R_0=R_{1,0}\oplus\dots \oplus R_{r,0},\qquad
R_{k-2}=R_{1,k-2}\oplus\dots \oplus R_{r,k-2},
$$
and
\begin{flalign*}
\ &\ & (R_{j,0}f)(\cdot):= (r_{j,0}*f)(\cdot),\quad
(R_{j,k-2}f)(\cdot):= (r_{j,k-2}*f)(\cdot),\ &\ & j\in\{1,\dots,r\}.
\end{flalign*}
Furthermore,  $R_0\in \Alg A_k$ by \eqref{Alg2inclusion} and
$R_{k-2}\in \Alg A_k$ by $(iii)$ . Thus (1) is proved.
\end{proof}
Combining Theorems \ref{AlgProp1} and  \ref{AlgProp2} we arrive at
\begin{theorem}
     Suppose
     $
     A_k(j):= \bigoplus_{i=1}^{n_j}\lambda_{ji} J_k^\alpha
     $
     is defined on $
     \bigoplus_{i=1}^{n_j}W_p^k[0,1],
     $
      $
     \ j\in\{1,\dots,r\}
     $
     and
     $
     A_k:=\bigoplus_{j=1}^r A(j)
     $
     is defined on
     $
     \bigoplus_{j=1}^r(\bigoplus_{i=1}^{n_j}W_p^k[0,1])
     $.
Let also
\begin{flalign*}
    \lambda_{ji}&=\lambda_{j1}/s_{ji}^{\alpha}, 1=s_{j1}\leqslant s_{j2}\leqslant \ldots\leqslant s_{jn_j},
                              &1&\leqslant j \leqslant r,\ \ 1\leqslant i \leqslant n_j,\\
\arg\lambda_{i1}&\ne \arg\lambda_{j1}           \pmod{2\pi},
                              &1&\leqslant i<j\leqslant r.\ \
\end{flalign*}
     Let also
\begin{flalign*}
 R:=\bigoplus_{j=1}^r\bigoplus_{i=1}^{n_j} R_{ji} \in
\biggl[\bigoplus_{j=1}^r\bigoplus_{i=1}^{n_j} W_p^k[0,1]\biggr],\ \
(R_{ji}f)(\cdot)=(r_{ji}*f)(\cdot),\ \ 1\leqslant j \leqslant r.
\end{flalign*}
Then the following are true:
\begin{itemize}
      \item[(1)]
      if\ \  $1\leqslant \alpha\leqslant k-1$, then
      \begin{equation*}
      \begin{split}
      \Alg  A_k=\Bigl\{&c\mathbb I +R:\ c\in  \mathbb C;\
      r_{j1}\in  W_p^{k-1}[0,1],\ 1\leqslant j\leqslant r;\\
      &r_{ji}(x)=s_{ji}^{-1}r_{j1}(s_{ji}^{-1}x),\ \ 1\leqslant j\leqslant r,
      \ \ 1\leqslant i\leqslant n_j;\\
      &r_{j1}^{(\alpha m-1)}(0)=
      (\lambda_{j1}\lambda_{11}^{-1})^{m}r_{11}^{(\alpha m-1)}(0),
      \ \ \ m\leqslant \Bigl[\frac{k-1}{\alpha}\Bigr],\ \ 1\leqslant j\leqslant r;\\
      &r_{j1}^{(l)}(0)=0,\ \ l\ne \alpha m-1,\
         m\leqslant \Bigl[\frac{k-1}{\alpha}\Bigr],\ \ 1\leqslant j\leqslant r\Bigr\};
      \end{split}
      \end{equation*}
      \item[(2)]
      if\ \  $2\leqslant k\leqslant \alpha+\frac 1p$, then
      \begin{equation*}
      \begin{split}
      \Alg  A_k=\bigl\{&c\mathbb I +R:\ c\in  \mathbb C;\
      r_{j1}\in  W_{p,0}^{k-1}[0,1],\ 1\leqslant j\leqslant r;\\
      &r_{ji}(x)=s_{ji}^{-1}r_{j1}(s_{ji}^{-1}x),\ \ 1\leqslant j\leqslant r,
      \ \ 1\leqslant i\leqslant n_j\bigr\}.
      \end{split}
      \end{equation*}
\end{itemize}
\end{theorem}
\begin{remark}
In this paper we do not consider questions about the reflexivity
of the operator $A_k$. Such results are contained in
\cite{I.Yu.Domanov and V.V.Surovtseva}.
\end{remark}
\subsection {The commutant  $\{A_k\}'$} \label{sec6}
 As in Section \ref{Preliminaries}  we define
operator $ L_{a}\in [W_p^k[0,1]]$ for $a\in (0,1]$ and $ L_{a}\in
[W_{p,0}^k[0,1],W_p^k[0,1]]$ for $a\in (1,\infty)$ by
\begin{equation}
      L_a:\ f(x)\to
                 g(x)=
                       \begin{cases}
                           f(ax)       &  0<a\leqslant 1, \\
                             \begin{cases}
                                   0,         & x\in [0,1-a^{-1}], \\
                                   f(ax-a+1), & x\in [1-a^{-1},1],
                             \end{cases} & a>1.
                       \end{cases}\label{definitionLa}
\end{equation}
Next we investigate solvability  of the equation
\begin{equation}\label{4.11}
      RJ^\alpha_k=c J^\alpha_kR
\end{equation}
in the space $X = W_p^k[0,1]$  and describe the set of its
solutions. The following proposition plays a crucial role in the
sequel. Its proof is based on Corollary \ref{mainfor01} and use
some ideas from \cite{I.Yu.Domanov and M.M.Malamud}.
\begin{proposition}\label{mainfor1}
      Let $c\in\mathbb C$ and let $R\in [X]$ be a
      solution  of  equation \eqref{4.11} where $X = W_p^k[0,1]$.
     Then
\begin{itemize}
     \item[(1)] \ \ If  $c\not\in \mathbb R_+$, then $R=0$;
     \item[(2)]
       \ \ If $0<c=a^\alpha\leqslant 1$, $a>0$, then $R\in L_a\{J_k^\alpha\}^{\prime}=
     \{J_k^\alpha\}^{\prime}L_a$,  that is,
     $$
     (Rf)(x)=\frac{d}{dx}\int_0^xr(x-t)f(at)\,dt, \qquad r\in W_p^k[0,1];
     $$
     \item[(3)]
     \ \ If  $1<c=a^\alpha$, $a>0$, then
     $
     R\in L_a\{J_k^\alpha\}^{\prime}
     $, that is,
\begin{align*}
(&Rf)(x)=\bigl(L_a\frac{d}{dx}(r*f)\bigr)(x)\\
     &=\begin{cases}
             0,                                 & x\in [0,1-a^{-1}], \\
        a^{-1}\frac{d}{dx}\int\limits_{0}^{ax-a+1}
                             r(ax-a+1-t)f(t)\,dt,\ \ r\in W_{p,0}^k[0,1], & x\in [1-a^{-1},1].
     \end{cases}
\end{align*}
\end{itemize}
\end{proposition}
\begin{proof}
Let $c\in\mathbb C$ and $ RJ_k^\alpha=c J_k^\alpha R $. Consider
the block matrix representations of the operators $J_k^\alpha$ and
$R$ with respect to the direct sum decomposition $
W_p^k[0,1]=W_{p,0}^k[0,1]\dotplus X_k $, where  $X_k:=\myspan\{1,x,\dots,
x^{k-1}\}$. Since $ W_{p,0}^k[0,1] \in \Lat J_k^\alpha $, one has
\begin{equation*}
      J_k^\alpha=
                 \left(
                  \begin{matrix}
                   J_{11}^\alpha  & J_{12}^\alpha\\
                   \mathbb O      & J_{22}^\alpha
                  \end{matrix}
                 \right),\qquad
      R         =
                  \left(
                   \begin{matrix}
                     R_{11} & R_{12}\\
                     R_{21} & R_{22}
                   \end{matrix}
                  \right).
\end{equation*}
Now  the equality $ RJ_k^\alpha=c J_k^\alpha R $ splits into
\begin{align}
       R_{11}J_{11}^\alpha   &=
      c  J_{11}^\alpha R_{11} +
      c J_{12}^\alpha R_{21},
      \label{equ7.3}
      \\
       R_{21}J_{11}^\alpha &=
       c J_{22}^\alpha R_{21},
      \label{equ7.4}
      \\
      R_{21}J_{12}^\alpha +
       R_{22}J_{22}^\alpha &=
      c J_{22}^\alpha R_{22},
      \nonumber
      \\
       R_{11}J_{12}^\alpha   +
       R_{12} J_{22}^\alpha  &=
       c J_{11}^\alpha R_{12} +
       c J_{12}^\alpha R_{22}.
       \nonumber
\end{align}
It is clear that $J_{22}^\alpha$ is a nilpotent operator on $X_k$
and consequently $J_{22}^{\alpha k}=0$. Therefore one derives from
\eqref{equ7.4} that $
 R_{21}J_{11}^{\alpha k}=c  J_{22}^{\alpha k}R_{21}=\mathbb O
$. It follows that $R_{21}=\mathbb O$ since  $\ran
J_{11}^{\alpha k} $ is dense in $W_{p,0}^k[0,1]$. Now equation
\eqref{equ7.3} takes the form $ R_{11}J_{11}^{\alpha}=c
J_{11}^{\alpha}R_{11} $, that is, $R_{11}$ intertwines the operators
$J_{11}^\alpha$ and $c J_{11}^\alpha$.

$(1)$  Let $c\not\in\mathbb R_+ $. Then Corollary \ref{mainfor01}
(i) yields $R_{11}=\mathbb O$. Furthermore, since $J^{\alpha k}x^m\in
W_{p,0}^k[0,1]$, $m\in\{0,\dots, k-1\} $, one has
\begin{equation*}
      0=
      R_{11}J_k^{\alpha k}x^m=
      R J_k^{\alpha k}x^m=
      cJ_k^{\alpha k}R x^m.
\end{equation*}
It follows that $Rx^m = 0$ for $m\in \{0,\dots, k-1\}$, hence
$R=\mathbb O$.

$(2)$ Let  $0<c=a^\alpha\leqslant 1$ for some $a>0$. Then
Corollary \ref{mainfor01} (ii) yields
$(R_{11}f)(x)=\frac{d}{dx}\int_0^xr(x-t)f(at)\,dt$, where $r\in L_{p'}[0,1]$. Let us prove that $r\in W_p^k[0,1]$. We
have
\begin{align*}
 a^{\alpha k} (J_k^{\alpha k}R\one)(x)&=(RJ_k^{\alpha
k}\one)(x)= (R_{11}J_k^{\alpha k}\one)(x)\\
 &= \frac{d}{dx}\int_0^xr(x-t)\frac{(at)^{\alpha k}}{\Gamma
(\alpha k+1)}\,dt= a^{\alpha k}(J^{\alpha k}r)(x).
\end{align*}
Hence $r=R\one\in W_p^k[0,1]$.

So, the operator $R_{11}$ defined on $W_{p,0}^k[0,1]$ admits a
continuation $T$ as an operator defined on $W_p^k[0,1]$ by
\begin{equation*}
       T:\ W_p^k[0,1]\rightarrow W_p^k[0,1],
       \ \qquad T:\ f(x)\rightarrow
       \frac{d}{dx}\int_0^xr(x-t)f(at)\,dt.
\end{equation*}
Since $
 T\upharpoonright W_{p,0}^k[0,1]=
 R\upharpoonright W_{p,0}^k[0,1]=
 R_{11}
$ and $
 J_k^{\alpha k}x^m\in W_{p,0}^k[0,1]
$ for $m\in\{0,\dots,k-1\}$, we obtain
\begin{equation*}
     J_k^{\alpha k}Tx^m=
     a^{-\alpha k}TJ_k^{\alpha k}x^m=
     a^{-\alpha k} RJ_k^{\alpha k}x^m=
     J_k^{\alpha k}Rx^m.
\end{equation*}
It follows that $Tx^m=Rx^m$ for $m\in\{0,\dots,k-1\}$. Thus
$R=T$.

$(3)$ Since $c=a^\alpha>1$, Corollary \ref{mainfor01} (ii) yields
\begin{align*}
(R_{11}f)(x)&=\bigl(L_a\frac{d}{dx}(r*f)\bigr)(x)\\
     &=\begin{cases}
             0,                                 & x\in [0,1-a^{-1}], \\
        a^{-1}\frac{d}{dx}\int\limits_{0}^{ax-a+1}
                             r(ax-a+1-t)f(t)\,dt, & x\in [1-a^{-1},1],
     \end{cases}
\end{align*}
where $r\in L_{p'}[0,1]$. Let us prove that $r\in
W_{p,0}^k[0,1]$.
\begin{align*}
a^{\alpha k} (J_k^{\alpha k}R\one)(x)&= (RJ_k^{\alpha
k}\one)(x)= (R_{11}J_k^{\alpha k}\one)(x)\\
&=
\begin{cases}
             0,                                      & x\in [0,1-a^{-1}], \\
        a^{-1}\frac{d}{dx}\int\limits_{0}^{ax-a+1}r(ax-a+1-t)
        \frac{t^{\alpha k}}{\Gamma(\alpha k+1)}  \,dt, & x\in [1-a^{-1},1],
     \end{cases}\\
&= \begin{cases}
    0,                                 & x\in [0,1-a^{-1}], \\
    a^{-1}\frac{d}{dx}(J^{\alpha k+1}r)(ax-a+1),
                                       & x\in [1-a^{-1},1],
\end{cases}\\
&=\begin{cases}
0,                             & x\in [0,1-a^{-1}], \\
(J^{\alpha k}r)(ax-a+1),       & x\in [1-a^{-1},1].
\end{cases}
\end{align*}
Hence
\begin{equation*}
(R\one)(x)=
\begin{cases}
0,                 & x\in [0,1-a^{-1}], \\
r(ax-a+1),         & x\in [1-a^{-1},1].
\end{cases}
\end{equation*}
Since $R\one\in W_p^k[0,1]$, it follows that
$r\in W_{p,0}^k[0,1]$.

So, the operator $R_{11}$ defined on $W_{p,0}^k[0,1]$ admits a
continuation $T$ on $W_p^k[0,1]$ defined by
\begin{equation*}
   (Tf)(x)=
     \begin{cases}
             0,                                 & x\in [0,1-a^{-1}], \\
        a^{-1}\frac{d}{dx}\int\limits_{0}^{ax-a+1}
                             r(ax-a+1-t)f(t)\,dt, & x\in [1-a^{-1},1].
     \end{cases}
\end{equation*}
 Since $
 T\upharpoonright W_{p,0}^k[0,1]=
 R\upharpoonright W_{p,0}^k[0,1]=
 R_{11}
$ and $
 J^{\alpha k}x^m\in W_{p,0}^k[0,1]
$ for $m\in\{0,\dots,k-1\}$, one deduces
\begin{equation*}
     J_k^{\alpha k}Tx^m=
     a^{-\alpha k}TJ_k^{\alpha k}x^m=
     a^{-\alpha k} RJ_k^{\alpha k}x^m=
     J_k^{\alpha k}Rx^m.
\end{equation*}
It follows that $Tx^m=Rx^m$ for $m\in\{0,\dots,k-1\}$. Thus
$R=T$.
\end{proof}
\begin{corollary}\cite[Theorem 3.4]{I.Yu.Domanov and M.M.Malamud} $R\in\{J_k^\alpha\}'$ if and only if
      \begin{equation*}
             (Rf)(x)=\frac{d}{dx}\int_0^xr(x-t)f(t)\,dt=r(0)f(x)+\int_0^xr'(x-t)f(t)\,dt,
             \ \ \ r\in W_p^k[0,1].
                 \end{equation*}
\end{corollary}
\begin{theorem}\label{pr7.3}
      Suppose
      $
      A_k=\bigoplus_{i=1}^n\lambda_i J_k^\alpha
      $
      is defined on
      $
     X^{(n)}= \bigoplus_{i=1}^n W_p^k[0,1]
      $
      and
\begin{equation*}
\lambda_i=
      \lambda_1/s_i^{\alpha},\qquad 1=
      s_1\leqslant s_2\leqslant \ldots\leqslant s_n,\qquad a_{ij}
      =s_i^{-1}s_j,\qquad 1\leqslant i,j\leqslant
      n.
\end{equation*}
      Then the commutant $\{A_k\}^{\prime}$ is of the form
      \begin{equation*}
             \{A_k\}^{\prime}=
             \{R:\ R=(R_{ij})_{i,j=1}^n,
             \ \ R_{ij}=L_{a_{ij}}K_{ij}\},
      \end{equation*}
      where
\begin{equation*}
      (K_{ij}f)(x)
=\frac{d}{dx}\int\limits_0^xk_{ij}(x-t)f(t)\,dt, \qquad
k_{ij}\in
\begin{cases}
W_p^k[0,1],    & a_{ij}\leqslant 1,\\
W_{p,0}^k[0,1],& a_{ij}>1.
\end{cases}
\end{equation*}
\end{theorem}
\begin{proof}
Let $R=(R_{ij})_{i,j=1}^n$ be the block matrix partition of the
operator $R$ with respect to the direct sum decomposition $
 X^{(n)} =\bigoplus_{i=1}^n W_p^k[0,1]$. Then the equality $RA_k=A_kR$ is
equivalent to the following system
\begin{equation*}
R_{ij} J_k^\alpha =
      \lambda_i\lambda_j^{-1} J_k^\alpha R_{ij}=
      (s_i^{-1}s_j)^\alpha    J_k^\alpha R_{ij}=
      a_{ij}^\alpha           J_k^\alpha R_{ij},
\qquad 1\leqslant i,\ j\leqslant n.
\end{equation*}
To complete the proof it remains to  apply  Proposition
\ref{mainfor1}.
\end{proof}
\begin{theorem}\label{pr7.1}
      Suppose
      $
      A_k=\bigoplus_{j=1}^r\lambda_j J_k^\alpha
      $
      is defined on
      $
      \bigoplus_{j=1}^r W_p^k[0,1]
      $
      and
      $
      \arg\lambda_i\ne \arg\lambda_j\pmod{2\pi}
      $
      for
      $1\leqslant i<j\leqslant r$.
      Then the commutant $\{A_k\}^{\prime}$
      splits,
      that is,
      \begin{equation*}
             \{A_k\}^{\prime}=
             \bigoplus_{j=1}^r\{\lambda_jJ_k^\alpha\}^{\prime}.
      \end{equation*}
\end{theorem}
\begin{proof}
Following the proof of Theorem \ref{pr7.3}, one arrives at the
relations
\begin{equation}
R_{ij} J_k^\alpha =
      \lambda_i\lambda_j^{-1} J_k^\alpha R_{ij},\qquad
 1\leqslant i,\ j\leqslant r.
      \label{lambdy}
\end{equation}
The latter results with $i=j$ yield $R_{ii}\in\{J_k^\alpha\}^{\prime}$ for
$i\in\{1,\dots,r\}$, hence by Proposition \ref{mainfor1} (2)
\begin{equation*}
R_{ii}:\
       f\rightarrow\frac{d}{dx}\int_0^xp_{ii}(x-t)f(t)\,dt,\qquad
       r_{ii}\in W_p^k[0,1],\qquad
       i\in\{1,\dots,r\}.
\end{equation*}
Since $\arg\lambda_i\ne \arg\lambda_j\pmod{2\pi}$  $(1\leqslant i<
j\leqslant r)$, it follows that $\lambda_i\lambda_j^{-1}\not\in\mathbb R_+$,
hence by Proposition \ref{mainfor1} (1)  $R_{ij}=0$ $(1\leqslant
i\ne j\leqslant r)$. This completes the proof.
\end{proof}
Combining  Theorems \ref{pr7.3} and \ref{pr7.1}, we arrive at
\begin{theorem}\label{th7.5}
Suppose
     $
     A_k(j):= \bigoplus_{i=1}^{n_j}\lambda_{ji} J_k^\alpha
     $
     is defined on $
     \bigoplus_{i=1}^{n_j}W_p^k[0,1]
     $
      $
     \ j\in\{1,\dots,r\}
     $
     and
     $
     A_k:=\bigoplus_{j=1}^r A(j)
     $
     is defined on
     $
     W=\bigoplus_{j=1}^r(\bigoplus_{i=1}^{n_j}W_p^k[0,1])
     $.
Let also
\begin{align*}
\arg\lambda_{j1}&=\arg\lambda_{ji}\pmod{2\pi},
     &1&\leqslant j \leqslant r,\qquad  1\leqslant i \leqslant n_j,\\
 \arg\lambda_{i1}&\ne \arg\lambda_{j1}\pmod{2\pi},
   &1&\leqslant i<j\leqslant r.
     \end{align*}
      Then
      $$
      \{A_k\}'=\bigoplus_{j=1}^r\{A_k(j)\}',
      $$
where the algebras  $\{A_k(j)\}'$ are described in Theorem \ref{pr7.3}.
\end{theorem}
\subsection {The double commutant $\{A_k\}''$}\label{sec777}
\begin{theorem}\label{bicargequal}
Suppose
      $
      A_k=\bigoplus_{i=1}^n\lambda_i J_k^\alpha
      $
      is defined on
      $
      W=\bigoplus_{i=1}^n W_p^k[0,1]
      $
      and
\begin{equation*}
\lambda_i=
      \lambda_1/s_i^{\alpha},\quad 1=
      s_1\leqslant s_2\leqslant \ldots\leqslant s_n,\qquad a_{ij}
      =s_i^{-1}s_j,\qquad   1\leqslant i,j\leqslant
      n.
\end{equation*}
Then
\begin{itemize}
\item[(1)]
\begin{equation*}
\begin{split}
          \{A_k\}^{\prime\prime}=
          \bigl\{&c\mathbb I +R :\ c\in \mathbb C,\ R=\diag (R_1,\dots ,R_n),\
          \ (R_if)(\cdot)=(r_i*f)(\cdot),\\
  &r_i(x)=s_i^{-1} r_1(s_i^{-1}x),\ \  r_i\in
W_p^{k-1}[0,1],\ \ 1\leqslant i \leqslant n\bigr\}.
\end{split}
\end{equation*}
\item[(2)]
      The dimension $d_{k,\alpha}$ of the quotient space
      $\{A_k\}^{\prime\prime}/\Alg A_k$ is
      $
       d_{k,\alpha}= k-1-[(k-1)/\alpha]
      $.
      In particular, $\Alg  A_k=\{A_k\}^{\prime\prime}$
      if and only if either $\alpha=1$ or $k=1$.
\end{itemize}
\end{theorem}
\begin{proof}
Let us set
\begin{equation*}
e_i:=(\underbrace{0,\dots,0, 1}_i,0,\dots, 0),\qquad
E_{ij}:=e_i^T e_j,\qquad 1\leqslant i,j\leqslant n.
\end{equation*}
Then Theorem \ref{pr7.3} implies
\begin{equation*}
\begin{split}
\{A_k\}'=\Alg\bigl\{&J_k\otimes E_{ii},\ \ 1\leqslant i\leqslant n;\\
&L_{a_{ij}}\otimes E_{ij},\ 1\leqslant j\leqslant i \leqslant n;\
L_{a_{ij}}J_k^k\otimes E_{ij},\ 1\leqslant i<j\leqslant n\bigr\}.
\end{split}
\end{equation*}
Since $\{\bigoplus_{i=1}^n\lambda_i J_k^\alpha\}''\subset
\bigoplus_{i=1}^n\{\lambda_i J_k^\alpha\}''$, it follows from \eqref{commutantsobolev}
\begin{equation*}
\begin{split}
          \{A_k\}^{\prime\prime}\subset
          \bigl\{T:=(c_1\mathbb I +R_1)\oplus \dots\oplus &(c_n\mathbb I +R_n) :\ c_i\in \mathbb
          C,\\
          &(R_if)(\cdot)=(r_i*f)(\cdot),\ \ r_i\in W_p^{k-1}[0,1]\bigr\}.
\end{split}
\end{equation*}
It is clear that $T(J_k\otimes E_{ii}) =(J_k\otimes E_{ii})T$ for
$i\in\{1,\dots, n\}$. It can easily be checked that
\begin{align}
T(L_{a_{ij}}\otimes E_{ij})&=(L_{a_{ij}}\otimes E_{ij})T,
 &1\leqslant j\leqslant i \leqslant n,\label{bikk1} \\
T(L_{a_{ij}}J_k^k\otimes E_{ij})&=(L_{a_{ij}}J_k^k\otimes
E_{ij})T,  &1\leqslant i<j\leqslant n\label{bikk2}
\end{align}
if and only if $c_1=\dots=c_n$ and $r_i(x)=s_i^{-1}
r_1(s_i^{-1}x)$ for $1\leqslant i\leqslant n$. Indeed,
\eqref{bikk1} and \eqref{bikk2} are equivalent to the first and the second of the following relations
\begin{align*}
(c_j\mathbb I + R_j)L_{a_{ij}}f&=L_{a_{ij}}(c_i\mathbb I +
R_i)f, &f\in W_p^k[0,1],&  &1\leqslant j\leqslant i
\leqslant n,\\
(c_j\mathbb I +
R_j)L_{a_{ij}}J_k^kf&=L_{a_{ij}}J_k^k(c_i\mathbb I + R_i)f,
&f\in W_p^k[0,1],&  &1\leqslant i< j \leqslant n,
\end{align*}
respectively. According to the definition of $L_{a_{ij}}$ (see
\eqref{definitionLa}), we obtain
\begin{equation}
c_jf(a_{ij}x)+\int\limits_0^xr_j(x-t)f(a_{ij}t)\,dt\\
=c_if(a_{ij}x)+\int\limits_0^{a_{ij}x}r_i(a_{ij}x-t)f(t)\,dt\label{bikk12}
\end{equation}
for $f\in W_p^k[0,1]$, $x\in [0,1]$ and $1 \leqslant j\leqslant i
\leqslant n$, and
\begin{equation}
 \begin{split}
&c_j(J_k^kf)(a_{ij}x-a_{ij}+1)+\int\limits_{1-a_{ij}^{-1}}^xr_j(x-t)(J_k^kf)(a_{ij}t-a_{ij}+1)\,dt\\
=&c_i(J_k^kf)(a_{ij}x-a_{ij}+1)+\int\limits_0^{a_{ij}x-a_{ij}+1}r_i(a_{ij}x-a_{ij}+1)(J_k^kf)(t)\,dt
\label{bikk22}
\end{split}
\end{equation}
for $f\in W_p^k[0,1]$, $x\in [1-a_{ij}^{-1},1]$ and $1 \leqslant
j< i \leqslant n$.

After simple computations with  \eqref{bikk12}-\eqref{bikk22},  we
get
\begin{align*}
\int\limits_0^x\bigl[r_j(x-t)-a_{ij}r_i(a_{ij}(x-t))\bigr]f(a_{ij}t)\,dt&=(c_i-c_j)f(a_{ij}x),\\
\int\limits_0^x\left[r_i(x-t)-a_{ij}^{-1}r_j(a_{ij}^{-1}(x-t))\right](J_k^kf)(t)\,dt&=(c_j-c_i)(J_k^kf)(x).
\end{align*}
Now it is easy to see that any of the latter  equations is equivalent
to $c_1=\dots=c_n$ and $r_i(x)=s_i^{-1}
r_1(s_i^{-1}x)$ for $i\in\{1,\dots,n\}$. Thus, $(1)$ is proved.

 $(2)$ It is clear that $W_p^{k-1}[0,1]\approx
W_{p,0}^{k-1}[0,1]\dotplus \myspan\{\frac{x^l}{l!} : l=1,\dots,k-2\}$.
Hence (1) implies that
\begin{equation}
\{A_k\}''\approx \mathbb C^1\dotplus W_p^{k-1}[0,1]\approx \mathbb
C^1\dotplus W_{p,0}^{k-1}[0,1]\dotplus \myspan\Bigl\{\frac{x^l}{l!} :
l=0,\dots,k-2\Bigr\}.\label{algbc1}
\end{equation}
Further, Theorem \ref{AlgProp2} yields $A_k$ is isomorphic
\begin{equation}
\Alg A_k\approx \mathbb C^1\dotplus
W_{p,0}^{k-1}[0,1]\dotplus\myspan\left\{\frac{x^{\alpha m-1}}{(\alpha
m-1)!} : 1\leqslant m\leqslant
\left[\frac{k-1}{\alpha}\right]\right\}.\label{algbc2}
\end{equation}
Combining \eqref{algbc1} with \eqref{algbc2} we easily arrive at $(2)$.
\end{proof}
\begin{theorem}\label{bicargnotequal}
 Suppose
       $
      A_k=\bigoplus_{j=1}^r\lambda_j J_k^\alpha
      $
      is defined on
      $
      \bigoplus_{j=1}^r W_p^k[0,1]
      $
      and
      $
      \arg\lambda_i\ne \arg\lambda_j\pmod{2\pi}
      $
      for
      $1\leqslant i<j\leqslant r$.      Then
\begin{itemize}
\item[(1)]
      $
      \{A_k\}^{\prime\prime}=\bigoplus_{j=1}^r\{J_k^\alpha\}^{\prime\prime}
      $.
\item[(2)]
      The dimension $d_{k,\alpha}$ of the quotient space
      $\{A_k\}^{\prime\prime}/\Alg A_k$ is
      $
      d_{k,\alpha}= rk-1-[(k-1)/\alpha]
      $.
      In particular, $\Alg  A_k=\{A_k\}^{\prime\prime}$ if and only if
      either
\begin{itemize}
      \item[(a)]
      $r=1$ and $\alpha=1$,
      or
      \item[(b)]
      $r=1$ and $k=1$.
\end{itemize}
\end{itemize}
\end{theorem}
\begin{proof}
$(1)$ is implied by Theorem \ref{pr7.1}. Furthermore, $(1)$ and Theorem
\eqref{AlgProp1} imply  that
\begin{align*}
\{A_k\}''&\approx\bigoplus_{j=1}^r(\mathbb C^1\oplus
W_p^{k-1}[0,1])\\
&\approx \mathbb C^r\
\dotplus\bigoplus_{j=1}^rW_{p,0}^{k-1}[0,1]\dotplus\bigoplus_{j=1}^r
 \myspan\Bigl\{\frac{x^l}{l!} :
l=0,\dots,k-2\Bigr\}\\
\Alg A_k&\approx\mathbb
C^1\dotplus\bigoplus_{j=1}^rW_{p,0}^{k-1}[0,1]\dotplus
\myspan\left\{\frac{x^{\alpha m-1}}{(\alpha m-1)!} : 1\leqslant
m\leqslant \left[\frac{k-1}{\alpha}\right]\right\}.
\end{align*}
Now it is easy to see that
$d_{k,\alpha}=r+r(k-1)-1-\left[\frac{k-1}{\alpha}\right]=rk-1-\left[\frac{k-1}{\alpha}\right]$.
Thus $(2)$ is proved.
\end{proof}
Combining Theorems \ref{bicargequal} and
\ref{bicargnotequal}, we obtain
\begin{theorem} Under the conditions of Theorem \ref{th7.5}, we have
      \begin{equation*}
      \{A_k\}''=\bigoplus_{j=1}^r\{A_k(j)\}'',
      \end{equation*}
where the algebras $\{A_k(j)\}''$ are described in Theorem
\ref{bicargequal}.
\end{theorem}
\begin{remark}
Recall that according to celebrated von Neumann theorem
$\{T\}^{\prime\prime}=\Alg T$ whenever $T$ is a normal operator.
B. Sz.-Nagy and C. Foias  \cite{B. Sz.-Nagy and C. Foias
1}-\cite{B. Sz.-Nagy and C. Foias 2} generalized this result
 to the wide class of accretive (disssipative) operators. In particular, this result holds
for the accretive operator $A=J\otimes B$ defined on
$L_2[0,1]\otimes \mathbb C^n$, where $B$ is a diagonal positive
matrix, $B = B^*>0$. By Theorem \ref{AlgProp2} this result remains
 also  valid for non-accretive operator $T:=A_k=J^{\alpha}_k\otimes B$
defined on $\bigoplus_{j=1}^nW_2^k[0,1]$, with  the same $B$ .
\end{remark}
\subsection
{ Invariant subspaces } \label{subsec4.2}
In \cite{I.Yu.Domanov and M.M.Malamud} we proved that  every
subspace invariant under
 $J_k^\alpha$ belongs either to the  "continuous chain"
$\Lat^c J_k^\alpha$ or to the "discrete chain" $\Lat^d J_k^\alpha$.
 It turns out that $\Lat^c J_k^\alpha$ does not depend on $\alpha$:
$\Lat^c J_k^\alpha=Lat^c J_k$ (see \eqref{neweq9}). We proved also
that the description of  $\Lat^d J_k^\alpha$  easily follows from
that of $\Lat J(0,k)^\alpha$. This  description is extracted from
Theorem \ref{theorem4.2}.

In this section we prove that every $A_k$-invariant subspace can
be decomposed into a direct sum of two
 invariant subspaces : the first one belongs to the "continuous part" of $\Lat A_k $ and the second one belongs to the "discrete part" of $\Lat A_k $.
We show also, that "continuous part"  does not depend on $\alpha$.
Moreover, a description of the "discrete part" is deduced from
Theorem \ref{theorem4.2}.

Let  $\chi_s$ stand for the  characteristic function of an
arbitrary nonempty subset $ S\subset\mathbb Z_n:=\{1,\dots, n\}$.
We denote by $P_S$  and $\widehat {P_S}$  the canonical
projections from $ \bigoplus_{j=1}^n W_p^{k_j}[0,1] $  and from $
\bigoplus_{j=1}^n C^{k_j} $ onto $ \bigoplus_{j=1}^n
\chi_{s}(j)W_p^{k_j}[0,1] $ and onto\\ $ \bigoplus_{j=1}^n
\chi_{s}(j)C^{k_j}$, respectively. Next we let
\begin{equation*}
      A_{k,S}:=\bigoplus_{j=1}^n \chi_{s}(j)
      \lambda_j J_{k_j}^\alpha\upharpoonright
      \ran P_S,\ \ \ \
      \widehat {A_{k,S}}:=\bigoplus_{j=1}^n \chi_{s}(j)
      \lambda_j J(0;k_j)^\alpha\upharpoonright \ran \widehat {P_S}
\end{equation*}
and denote by $\pi_S$ the quotient mapping from $\ran P_S $ onto
$\ran\widehat {P_S}$.
\begin{theorem}\label{th4.3}
      Suppose
      $
      A_k=\bigoplus_{j=1}^n\lambda_j J_{k_j}^\alpha
      $
      is defined on
      $
      \bigoplus_{j=1}^n W_p^{k_j}[0,1]
      $
      and
      $
      \arg\lambda_i\ne \arg\lambda_j\pmod{2\pi}
      $
      for
      $1\leqslant i<j\leqslant n$.
      Then $E\in \Lat A_k$ if and only if there exists $ S\subset\mathbb Z_n$
      and $a_1,\dots,a_n\in [0,1]$ such that
      \begin{equation*}
      E=\Lat  A_{k,S}\bigoplus_{j=1}^n\chi_{s^c}(j)E_{a_j,0}^{k_j},
      \end{equation*}
            where
      \begin{equation}
         \Lat  A_{k,S}=
         \bigcup_M\pi_S^{-1}\left\{[M,(\widehat {A_{k,S}})^{-1}M]\
         :\ M\in \Lat  \widehat {A_{k,S}}\upharpoonright\widehat {A_{k,S}}M\right\}
         \label{1111}
      \end{equation}
      and $S^c$ is the complement for $S$ in
      $\mathbb Z_n$ $(S\cup S^c=\mathbb Z_n)$.
        Here $[M,(\widehat {A_{k,S}})^{-1}M]$ is a closed interval
        in the lattice of all subspaces of  $\ran \widehat {P_S}$.
      Each interval satisfies the equation
      \begin{equation}
             \dim (\widehat {A_{k,S}})^{-1}M-\dim M=
             \sum_{j\in S} \min\bigl\{-[-\alpha ],k_j\bigr\}.
             \label{2222}
      \end{equation}
\end{theorem}
\begin{proof}
For every $E\in \Lat  A_k$, we put $j$ in $S:=S_E$ if
$P_jE\not\subset W_{p,0}^{k_j}[0,1]$ and put $j$ in $S^c$
otherwise.
  Next we introduce the subspaces
  $
  E_S:=\myspan  \{A_{k,S}^mP_SE :\ m\geqslant 0\}
  $
  and
  $
  E_{S^c}:=\myspan  \{A_{k,S^c}^mP_{S^c}E :\ m\geqslant 0\}
  \subset\bigoplus_{j=1}^n \chi_{s^c}(j)W_{p,0}^{k_j}[0,1]
  $.
It is clear that $E\subset E_{S}\oplus E_{S^c}$.

Let $ M=\max\limits_{1\leqslant j\leqslant n}k_j $.
  Then the subspace $F:=\overline{A_k^M E}$ is invariant for the operator
  $A_{k,0}:= A_k\upharpoonright\bigoplus_{j=1}^n W_{p,0}^{k_j}[0,1]$
  and, by Theorem \ref{split}, $F=\bigoplus_{j=1}^n E_{a_j,0}^{k_j}$
  for some $a_j\in[0,1]$.
By the construction of $S$, it is clear that $a_j=0$ for $j\in S$ and
hence
\begin{equation}
      F=\Biggl(\bigoplus_{j=1}^n \chi_{s}(j)W_{p,0}^{k_j}[0,1]\Biggr)
      \cup\Biggl(\bigoplus_{j=1}^n \chi_{s^c}(j)E_{a_j,0}^{k_j}\Biggr).
      \label{equ4.2}
\end{equation}
  It is clear that $E\supset F\supset E_{S^c}$.
Hence ${E}\supset P_{S^c}E$ and, therefore, $E\supset P_{S}E$.
  The latter inclusion yields $E\supset E_{S}$ and consequently
  $E$ splits : $E=E_{S}\oplus E_{S^c}$.

In turn, by Theorem \ref{split}, $E_{S^c}$ splits:
$E_{S^c}=\bigoplus_{j=1}^n \chi_{s^c}(j)E_{a_j,0}^{k_j}$.
  On the other hand, combining \eqref{equ4.2} with the relations
  $E=E_{S}\oplus E_{S^c}\supset F$, one gets
  $E_S\supset\bigoplus_{j=1}^n\chi_{s}(j)W_{p,0}^{k_j}[0,1]$.
Therefore,  $\pi_S(E_S)\in \Lat \widehat {A_S}$.
  Since the quotient map $\pi_S$ establishes a bijective correspondence
  between $E_S\in \Lat  {A_S}$ with
  $E_S\supset\bigoplus_{j\in S} W_{p,0}^{k_j}[0,1]$ and
  $\pi_S(E_S)$, one derives $E_S=\pi_S^{-1}(\pi_SE_S)$.
One completes the proof by applying Theorem \ref{theorem4.2}.
  Furthermore,  relations \eqref{1111}  and \eqref{2222} are implied
  by the relations \eqref{eqno(2.3)} and \eqref{eqno(2.4)}, respectively.
\end{proof}
\begin{corollary}\cite{I.Yu.Domanov and M.M.Malamud}\label{scalth2.5.}
      Let $\pi$ be the quotient map
      \begin{equation*}
             \pi :\ W_p^k[0,1]\rightarrow X_k:=
             \ W_p^k[0,1]/W_{p,0}^k[0,1]
      \end{equation*}
      and $\widehat {J_k^\alpha}$ be the quotient operator on $X_k$.
      Then
      $\Lat  J_k^\alpha=\Lat ^c J_k^\alpha\cup \Lat ^d J_k^\alpha$,
      where
\begin{itemize}
\item[(a)]
      \begin{equation*}
              \Lat ^c J_k^\alpha=
             \bigl\{E_{a,0}^k :\ 0\leqslant  a\leqslant 1 \bigr\},\ E_{a,0}^k:=
             \bigl\{f\in W_{p,0}^k[0,1]:\ f(x)=0,\
              x\in[0,a]\bigr\}
      \end{equation*}
      is the "continuous part" of $\Lat  J_k^\alpha$;
\item[(b)]
      \begin{equation*}
             \Lat ^d J_k^\alpha=
             \pi^{-1}(\Lat  \widehat {J_k^\alpha})=
             \bigcup_M\pi^{-1}
             \left\{
             [M,(\widehat {J_k^\alpha})^{-1}M]:\
             \ M\in \Lat
             (\widehat {J_k^\alpha}\upharpoonright \widehat {J_k^\alpha} M)
             \right\}
      \end{equation*}
      is the "discrete part" of $\Lat  J_k^\alpha$.

      Here $[M,(\widehat {J_k^\alpha})^{-1}M]$ is a closed interval
      in the lattice of all subspaces of $X_k$.
      Each interval satisfies the equation
      \begin{equation*}
             \dim(\widehat {J_k^\alpha})^{-1}M-\dim M=d,
      \end{equation*}
      where $d=\min\{-[-\alpha],k\}$.
\end{itemize}
\end{corollary}
\begin{corollary}\cite{I.Yu.Domanov and M.M.Malamud}
Operator $J_k^\alpha$ is unicellular if and only if either
$\alpha=1$ or $k=1$.
\end{corollary}
\begin{example}\label{ex4.4}
    Suppose that the operator
    $
    A=\lambda_1J_{k_1}^\alpha\oplus \lambda_2J_{k_2}^\alpha
    $
    $(\arg\lambda_1\ne \arg\lambda_2)\pmod{2\pi}$
is defined on
    $
    W_p^{k_1}[0,1]\oplus W_p^{k_2}[0,1]
    $.
    By Theorem \ref{th4.3}, one has the following description of its lattice of invariant subspaces :
   \begin{align*}
             \Lat  A=
             &\bigcup_{[a_1,a_2]\in [0,1]\times [0,1]}
             (E_{a_1,0}^{k_1}\oplus E_{a_2,0}^{k_2})
             \cup\bigcup_{a\in [0,1]}\pi_{\{1\}}^{-1}
             (\Lat \widehat {A_{\{1\}}})\oplus E_{a,0}^k\\
             &\cup\bigcup_{a\in [0,1]}
             E_{a,0}^k\oplus\pi_{\{2\}}^{-1}(\Lat \widehat {A_{\{2\}}})
             \cup\bigcup\pi_{\{1,2\}}^{-1}(\Lat \widehat {A_{\{1,2\}}}),
    \end{align*}
    where lattices
      $
      \pi_{\{1\}}^{-1}(\Lat  A_{\{1\}})=\Lat ^d J_{k_1}^\alpha
      $
    and
      $
      \pi_{\{2\}}^{-1}(\Lat  A_{\{2\}})=\Lat ^d J_{k_2}^\alpha
      $
    are described in Corollary \ref{scalth2.5.}.
    For example, if
      $
      k_1=1,\ k_2=2,\ \lambda_1=i,\ \lambda_2=1
      $
    and
      $\alpha=1$,
    one has
      $
      \pi_{\{1\}}^{-1}(\Lat \widehat {A_{\{1\}}})=
      \Lat ^d J_{1}^1=
      W_{p,0}^1[0,1]\cup W_p^1[0,1]
      $,\
      $
      \pi_{\{2\}}^{-1}(\Lat \widehat {A_{\{2\}}})=
      \Lat ^d J_{2}^1=
      W_{p,0}^2[0,1]\cup E_1^2\cup W_p^2[0,1]
      $.
    It is easily seen that
      $
      \widehat {A_{\{1,2\}}}=
      0\oplus J(0;2)
      $,
    hence,
      $
      \widehat {A_{\{1,2\}}}\upharpoonright \ran
      (\widehat {A_{\{1,2\}}}):\ e_3\rightarrow 0
      $
    $($here $\{e_1,e_2,e_3\}$ is the standard basis in $\mathbb C^3)$.
    Thus, by Theorem \ref{theorem4.2},
     \begin{align*}
             \Lat  \widehat {A_{\{1,2\}}}&=
             \bigcup_{M\subset \{e_3\}}
             [M, (\widehat {A_{\{1,2\}}})^{-1}M]=
             [0, \{e_1,e_3\}]\cup[\{e_3\}, \{e_1,e_2,e_3\}]\\
             &=
             \{0\}\cup\bigcup_{\alpha,\beta\in \mathbb C}
             \{\alpha e_1+\beta e_3\}
             \cup\bigcup_{\alpha,\beta\in \mathbb C}
             \{\alpha e_1+\beta e_2,e_3\}
             \cup\{e_1,e_2,e_3\}\\
             &\approx
             \{0\}\cup\bigcup_{\alpha,\beta\in \mathbb C}
             \{(\alpha,\beta x)\}\cup\bigcup_{\alpha,\beta\in \mathbb C}
             \{(\alpha ,\beta),(0,x)\}\cup\{(1,0),(0,1),(0,x)\}.
      \end{align*}
    Hence
      \begin{align*}
             \pi_{\{1,2\}}^{-1}(&\Lat \widehat {A_{\{1,2\}}})=
             \bigl(W_{p,0}^1[0,1]\oplus W_{p,0}^2[0,1]\bigr)\\
             &\cup\bigcup_{\alpha,\beta\in\mathbb C}
             \bigl\{\{f_1,f_2\}:\ f_1\in W_p^1[0,1],f_2\in E_1^2,
             \alpha f_1(0)+\beta f_2'(0)=0\bigr\}\\
             &\cup\bigcup_{\alpha,\beta\in \mathbb C}
             \bigl\{\{f_1,f_2\}:\ f_1\in W_p^1[0,1],f_2\in W_p^2[0,1],
             \alpha f_1(0)+\beta f_2(0)=0\bigr\}\\
             &\cup\bigl(W_{p}^1[0,1]\oplus W_{p}^2[0,1]\bigr).
      \end{align*}
\end{example}
\begin{remark}
\begin{itemize}
     \item[(i)] An alternative description of $\Lat ^d J_k^\alpha$ might be obtained
     from the Halmos description of $\Lat  T$ for $T\in [\mathbb C^n]$
     $($see Theorem \ref{Halmos}$)$.
    \item[(ii)] A quite different proof of the description  of $\Lat  J_k$ has
    been originally obtained  by E.Tsekanovskii \cite{E.R.Tsekanovskii}.
\end{itemize}
\end{remark}
\subsection
{Hyperinvariant subspaces}\label{hyperinvAk}
To present a description of $\Hyplat A_k$ we keep  the notation from
Subsection \ref{subsec4.2}.
\begin{theorem}\label{pr7.4}
      Let the conditions of Theorem \ref{pr7.3} hold. Then
\begin{equation*}
\Hyplat    A_k=
             \bigcup_{S\subset \mathbb Z_n}\{E_{S^c}\oplus E_S\}.
\end{equation*}
Here
\begin{itemize}
\item[(a)]
      "the continuous part" $E_{S^c}$ is of the form
\begin{equation*}
              E_{S^c}=
             \biggl\{\bigoplus_{j=1}^n\chi_{S^c}(j)E_{a_j,0}^k:\ a=\{a_j\}_{j\in S^c}
             \in P(\{s_j\}_{j\in S^c}) \biggr\},
\end{equation*}
where
\begin{align*}
P(\{s_i\}_{i\in S^c}):=P(s_{n_1},\dots,s_{n_{|S^c|}})=
      \Bigl\{(a_{n_1},\ldots,a_{n_{|S^c|}})\in\square_{|S^c|}:\\
      s_{n_j}a_{n_{j+1}}\leqslant s_{n_{j+1}}a_{n_j}\leqslant s_{n_{j+1}}-s_{n_j}+s_{n_j}a_{n_{j+1}},\
      1\leqslant j\leqslant  |S^c|-1
      \Bigr\}.
\end{align*}
\item[(b)]
      "the discrete part" $E_S$ is of the form
      $
      E_S= \bigoplus_{j=1}^n \chi_{S}(j)E_{l_j}^k
      $,
      where $1\leqslant l_j\leqslant k-1$ and
      $l_j\leqslant l_i$ if $s_j\leqslant s_i$ for $1\leqslant i,j\leqslant n$;
\end{itemize}
      In particular, if $\lambda_1=\dots=\lambda_n$, then
      \begin{equation*}
             \Hyplat    A_k=
             \bigcup_{S\subset \mathbb Z_n,\ a\in [0,1],\ 1\leqslant l\leqslant k-1}
             \Biggl\{
              \bigoplus_{j=1}^n\chi_{s}(j) E_l^k
              \bigoplus_{i=j}^n \chi_{s^c}(j)E_{a,0}^k
            \Biggr\}.
      \end{equation*}
\end{theorem}
\begin{proof}
It is clear that $\Hyplat A_k=\Hyplat
(\bigoplus_{j=1}^n\lambda_jJ_k^\alpha)\subset \bigoplus_{j=1}^n\Hyplat
\lambda_jJ_k^\alpha=\bigoplus_{j=1}^n\Lat\lambda_jJ_k$. Hence if
$E\in \Hyplat A_k$ then  $E=\bigoplus_{j=1}^nE_j$, where $E_j\in \Lat
J_k$. For each $E_j\in \Lat J_k$ $(1\leqslant j\leqslant n)$ we
put $j$ in $S$ if $E_j\in \Lat^dJ_k\backslash W_{p,0}^k[0,1]$ and
put $j$ in $S^c$ otherwise (i.e., if $E_j\in \Lat^cJ_k$). Thus
$E=E_S\oplus E_{S^c}$, where
$E_{S^c}=\bigoplus_{j=1}^n\chi_S(j)E_{a_j,0}^k$ and
$E_S=\bigoplus_{j=1}^n\chi_{S^c}(j)E_{l_j}^k$. Now $E_{S^c}$ is
described in Theorem \ref{theorem29} and Corollary
\ref{corollary32}. Let us prove that
$E_S=\bigoplus_{j=1}^n\chi_{S^c}(j)E_{l_j}^k\in \Hyplat A_k$ if and
only if $l_j\leqslant l_i$ whenever  $s_j\leqslant s_i$ for
$1\leqslant i,j\leqslant n$.

Let $s_j\leqslant s_i$ and $P\in\{A_k\}^{\prime}$ be such that the
block matrix partition of the operator $P$ with respect to the
direct sum decomposition $\bigoplus_{j=1}^nW_p^k[0,1]$ contains the
only non-zero element $P_{ij}:=L_{a_{ij}}$. Then the inclusion
$PE_S\subset E_S$ yields $E_{l_j}=P_{ij}E_{l_j}\subset E_{l_i}$.
So $s_j\leqslant s_i$ yields $E_{l_j}\subset E_{l_i}$ or
$l_j\leqslant l_i$.

The opposite statement may be obtained using routine matrix
calculations, which we omit.
\end{proof}
\begin{theorem}\label{pr7.2}
      Under the conditions of Theorem \ref{pr7.1},
      the lattice $\Hyplat  A_k$ splits :
\begin{equation*}
 \Hyplat   A_k = \bigoplus_{j=1}^n \Hyplat
\lambda_j J_k^\alpha= \bigoplus_{j=1}^n \Lat  J_k.
\end{equation*}
\end{theorem}
\begin{remark}\label{rem7.3}
It is well known $($see \cite{J.B.Conway and B. P.Y.Wu}$)$ that
for two bounded operators $T_1$ and $T_2$ the splitting of $\Lat
(T_1\oplus T_2)$ implies the splitting of $\Hyplat  (T_1\oplus
T_2)$. In other words, the relation $ \Lat (T_1\oplus T_2)=\Lat
T_1\oplus \Lat  T_2 $ yields the relation $ \Hyplat  (T_1\oplus
T_2)=\Hyplat   T_1\oplus \Hyplat   T_2 $. Theorem \ref{pr7.2}
demonstrates that the converse implication is not true. Nevertheless the
converse implication  is true for $C_0$ contractions $T_1$ and
$T_2$ defined on  Hilbert space $($\cite{J.B.Conway and B. P.Y.Wu}$)$.
\end{remark}
Summing up Theorems \ref{pr7.4} and \ref{pr7.2}, we obtain
\begin{theorem}\label{th7.6}
      Under the conditions of Theorem \ref{th7.5}, we have
      \begin{equation*}
             \Hyplat   A_k=
             \bigoplus_{j=1}^r\Hyplat A_k(j),
      \end{equation*}
      where the lattices $\Hyplat A_k(j)$
      are described in Corollary \ref{pr7.4}.
\end{theorem}
\subsection {Cyclic subspaces} \label{sec5}
Some results of this subsection were announced in \cite{I.Yu. Domanov 2}.
First, we present the following simple
\begin{lemma}\label{l5.5}
      Let $A\in [\mathbb C^k]$, $\sigma (A)=\{0\}$ and
      $P_{\ker A^*}$ be the orthoprojection from $\mathbb C^k$ onto $\ker    A^*$.
      Then
\begin{itemize}
\item[(1)]
     $\mu_{A}=\disc A=\dim(\ker A^*)=\dim(\ker A)$;
\item[(2)]
     $E\in \Cyc  A$ if and only if $PE=\ker A^*$.
\end{itemize}
\end{lemma}
\begin{proof}
\emph{Necessity.} Note that $\myspan \{E,\ran A\}\supset\myspan
\{A^jE:\ j\geqslant 0\}$ and $(\mathbb I_k-P_{\ker A^*})E=P_{\ran
A}E\subset \ran A$. Therefore, since $E\in \Cyc  A$, we have
\begin{align*}
\mathbb C^k&=\myspan  \{A^jE :\
j=0,1,\dots,k-1\}\subset \myspan \{P_{\ker
A^*}E,(\mathbb I_k -P_{\ker A^*})E, \ran A\}\\
&=\myspan
\{P_{\ker A^*}E, \ran A\}\subset \myspan
\{\ker A^*, \ran A\}= \ker A^* \oplus \ran A=\mathbb C^k.
\end{align*}
Hence $P_{\ker A^*}E=\ker    A^*$.

\emph{Sufficiency.}
 Let $PE=\ker   A^*$. Then
\begin{equation*}
\mathbb C^k=\myspan  \{P_{\ker A^*}E,\ran A\}\subset \myspan
\{E,(\mathbb I_k-P_{\ker A^*})E,\ran A\}=\myspan \{E,\ran A\}.
\end{equation*}
Applying the operator $A^j$, we obtain
 $\ran A^j=\myspan \{A^jE,\ran A^{j+1}\}$ $(1\leqslant j\leqslant k-1)$.
 Hence
\begin{equation*}
\mathbb C^k=\myspan  \{E,\ran A\}=\myspan \{E,AE,\ran A^2 \}=\dots
=\myspan \{E,\dots,A^{k-1}E\}.
\end{equation*}
It means $E\in \Cyc  A$.
\end{proof}
For every system $
\phi=\{\overrightarrow{\phi_l}\}_1^N$, $\overrightarrow{\phi_l}\in
\mathbb C^n$, we denote by $W(\phi)$ the $n\times N$ matrix
consisting of the columns $ \overrightarrow{\phi_l}:\ W(\phi)=
(\overrightarrow{\phi_1},\dots ,\overrightarrow{\phi_N}) $.
\begin{corollary}\label{cor5.6}
      Suppose that
      $
      A=\bigoplus_{j=1}^n\lambda_j J(0;k_j)^{\alpha}
      $
      is defined on
      $
      \bigoplus_{j=1}^n\mathbb C^{k_j}
      $
      and
      $
      m_j:= \min(-[-\alpha],k_j)
      $ for $1\leqslant j\leqslant n$.
      Then
\begin{itemize}
\item[(1)]
       $\mu_A=\disc A=\sum_{j=1}^n m_j$;
\item[(2)]
        the following system
         \begin{equation*}
\overrightarrow{\phi_l}=
               \mathrm{col}
              (
                \phi_{l11},\dots ,\phi_{l1k_1},
                \phi_{l21},\dots ,\phi_{l2k_2},
                \dots \dots\
                \phi_{ln1},\dots ,\phi_{lnk_n}
              ),\qquad
1\leqslant l\leqslant N
         \end{equation*}
     generates a cyclic subspace for the operator
      $
      A
      $

      if and only if
\begin{itemize}
      \item[(1)] $N\geqslant \sum_{j=1}^n m_j$;

      \item[(2)] the matrix
      $W_0=P_{\ker A^*}W(\phi)$ is of maximal rank,
      that is, $\rank W_0=\sum_{j=1}^n m_j$.
\end{itemize}
\end{itemize}
\end{corollary}
\begin{theorem}\label{th5.6}
      Suppose
      $
      A_k=\bigoplus_{j=1}^n\lambda_j J_{k_j}^{\alpha}
      $
      is defined on
      $\bigoplus_{j=1}^n W_p^{k_j}[0,1]$
      and \ $m_j:= \min(-[-\alpha],k_j)$ for $1\leqslant j\leqslant n$.
      Then
\begin{itemize}
\item[(1)]
       $\mu_{A_k}=\disc A_k=\sum_{j=1}^n m_j$;
\item[(2)]
       the system $\{f_l(x)\}_{l=1}^N$ of vectors
      ${f_l(x)}=\{ f_{l1}(x),\dots ,f_{ln}(x)\}$
      generates a cyclic subspace for  $A_k$
      if and only if the following conditions hold
\begin{itemize}
      \item[(i)] $N\geqslant \sum_{j=1}^n m_j$;

      \item[(ii)] the matrix
      \begin{equation*}
          W(0)=
            \left(
             \begin{matrix}
             f_{11}(0) & f_{21}(0) & \dots & f_{N1}(0)         \\
             f_{11}'(0)&     f_{21}'(0) &\dots &f_{N1}'(0)     \\
             \vdots     & \vdots     &  & \vdots             \\
             f_{11}^{(m_1-1)}(0) & f_{21}^{(m_1-1)}(0) &
             \dots & f_{N1}^{(m_1-1)}(0)                       \\
             \vdots     & \vdots     &  & \vdots             \\
             f_{1n}(0) & f_{2n}(0) & \dots & f_{Nn}(0)         \\
             f_{1n}'(0)& f_{2n}'(0)& \dots & f_{Nn}'(0)        \\
             \vdots     & \vdots     & & \vdots             \\
             f_{1n}^{(m_n-1)}(0) & f_{2n}^{(m_n-1)}(0) &
             \dots & f_{Nn}^{(m_n-1)}(0)                       \\
           \end{matrix}
          \right)\
      \end{equation*}
      is of maximal rank, i.e., $\rank W(0)=\sum_{j=1}^n m_j$.
\end{itemize}
\end{itemize}
\end{theorem}
\begin{proof}
It is clear that $E\in \Cyc A_k$ implies $\pi E\in \Cyc \widehat
{A_k}$. To prove the converse assertion we choose a subspace
$E\subset \bigoplus_{j=1}^n W_p^{k_j}[0,1]$ such that $\pi E\in \Cyc
\widehat {A_k}$ and denote by $F:=\myspan  \{A^jE:\ j\geqslant
0\}$. Since $ \pi F=\bigoplus_{j=1}^n\mathbb C^{k_j}$, one gets that
$ F\supset \bigoplus_{j=1}^nW_{p,0}^{k_j}[0,1]$. Therefore, just in
the same way as in Theorem \ref{th4.3}, we obtain that $
F=\pi^{-1}(\pi F)=\pi^{-1}(\bigoplus_{j=1}^n\mathbb C^{k_j})=
\bigoplus_{j=1}^n W_p^{k_j}[0,1]$, that is, $E\in \Cyc A_k$. To
complete the proof it suffices to apply Corollary \ref{cor5.6}.
\end{proof}
\begin{remark}\label{rem1}
     For $\alpha=1$ and $k_1=\dots = k_n=:k\geqslant 1$, that is, for the operator
     $
     A_k=\bigoplus_{j=1}^n\lambda_jJ_{k}
     $,
     Theorem \ref{th5.6} has been established in \cite{I.Yu. Domanov}
     by another method.

     We emphasize that the description of the
     set $\Cyc  A_{k,0}$ essentially differs from that of $\Cyc
     A_k$.
Namely, in contrast to the operator $A_{k,0}$,
     the description of the set $\Cyc  A_k$ does not depend
     on the choice of $\lambda_j$.
\end{remark}
Summing up, we obtain a  description of the  cyclic subspaces for
the operator $ A=\bigoplus_{j=1}^m\lambda_j
J_{k_j}^\alpha\oplus\bigoplus_{j=m+1}^n \lambda_jJ_{k_j,0}^\alpha
$ acting on the mixed space $ \bigoplus_{j=1}^m
W_p^{k_j}[0,1]\oplus \bigoplus_{j=m+1}^n W_{p,0}^{k_j}[0,1]$.
\begin{theorem}\label{th5.7}
Suppose that the operators
\begin{equation*}
          A_k(1):=\bigoplus_{j=1}^m\lambda_j J_{k_j}^\alpha
         ,\ \
          A_{k,0}(1):=\bigoplus_{j=1}^m\lambda_j J_{k_j,0}^\alpha
\end{equation*}
and
\begin{equation*}
          A_{k,0}(2):=\bigoplus_{j=m+1}^n\lambda_j J_{k_j,0}^\alpha
          ,\ \text{ and }\
          A:=A_k(1)\oplus A_{k,0}(2)
\end{equation*}
      are defined on
\begin{equation*}
         X(1):=\bigoplus_{j=1}^m W_p^{k_j}[0,1]
         ,\ \
          X_{0}(1):=\bigoplus_{j=1}^m W_{p,0}^{k_j}[0,1]
\end{equation*}
and
\begin{equation*}
          X_{0}(2):=\bigoplus_{j=m+1}^n W_{p,0}^{k_j}[0,1]
         ,\ \text{ and } \
         X:=X(1)\oplus X_{0}(2),\ \
\end{equation*}
respectively. Furthermore, let $P(1)$ be the canonical projection from
$X=X(1)\oplus X_{0}(2)$ onto $X(1)$. Then
\begin{itemize}
\item[(1)]
       $\mu_{A}=\max \{\mu_{A_k(1)},\mu_{A_{k,0}(1)\oplus
       A_{k,0}(2)}\}$;
\item[(2)]
       $E\in \Cyc A$ if and only if
\begin{itemize}
       \item[(a)] $ P(1)E\in \Cyc A_k(1)$,
       \item[(b)] $\overline{A^ME}\in\Cyc(A_{k,0}(1)\oplus
        A_{k,0}(2))$,
      where $M:=\max\limits_{1\leqslant j\leqslant m}k_j$.
      Furthermore,  the set
      $
      \Cyc(A_{k,0}(1)\oplus A_{k,0}(2))
      $
      is described in Theorems \ref{th4.1Lp} and
       \ref{th5.3}  and the
      set  $\Cyc A_k(1)$  is described in Theorem \ref{th5.6}.
\end{itemize}
\end{itemize}
      \end{theorem}

  We express our gratitude
to Professor Pei Yuan Wu for giving precise references concerning
Theorem \ref{Halmos}. We are also grateful to the referee for a
number of helpful suggestions for improvement in the article.

I.Yu. Domanov\\
Institute of Applied Mathematics and Mechanics\\
Roza-Luxemburg 74\\
Donetsk 83114\\
Ukraine\\
domanovi$@$yahoo.com

M.M. Malamud\\
Institute of Applied Mathematics and Mechanics\\
Roza-Luxemburg 74\\
Donetsk 83114\\
Ukraine\\
mmm$@$telenet.dn.ua
\end{document}